\documentclass[12pt,reqno]{amsart}

\usepackage{fullpage}

\usepackage{cite}
\usepackage{amsmath,amssymb,amsthm}
\usepackage{color}
\usepackage[margin=2.2cm]{geometry}

\newtheorem{theo}{Theorem}[section]
\newtheorem{lemm}[theo]{Lemma}
\newtheorem{defi}[theo]{Definition}
\newtheorem{coro}[theo]{Corollary}
\newtheorem{prop}[theo]{Proposition}
\newtheorem{rema}[theo]{Remark}
\numberwithin{equation}{section}

\def\Re{\mathrm{Re}\,}
\def\Im{\mathrm{Im}\,}
\def\Rnum{\mathbb{R}}
\def\Cnum{\mathbb{C}}
\def\Nnum{\mathbb{N}}
\def\d{\mathrm{d}}
\def\sgn{\mathrm{sgn}}

\def\supp{\mathrm{supp}}

\def\C{C}

\def\c{\tilde C}
\def\cc{\hat C}

\def\Ref#1{Ref.~\cite{#1}}
\def\thmname#1{$\it (${\bf #1}$\it )$}

\begin{document}

\allowdisplaybreaks[3]

\title{Local well-posedness and blow-up for\\ a family of $U(1)$-invariant peakon equations}

\author{
Stephen Anco$^1$,
Huijun He$^{1}$,
Zhijun  Qiao$^2$\\\\
$^1$ Department of Mathematics and Statistics,\\
Brock University,  St. Catharines, Ontario, L2S 3A1, Canada\\
$^2$ School of Mathematical and  Statistical Sciences,\\
 University of Texas at Rio Grande Valley (UTRGV),
Edinburg, TX, 78539, USA
}

\date{}

\begin{abstract}
The Cauchy problem for a unified family of integrable $U(1)$-invariant peakon equations from the NLS hierarchy is studied.
As main results,
local well-posedness is proved in Besov spaces,
and blow-up is established through use of an $L^1$ conservation law.
\\

\noindent
2000 Mathematics Subject Classification: 35G25, 35L05, 35Q53

\noindent
\textit{Keywords}:
$U(1)$-invariant,
peakon,
NLS hierarchy,
modified-CH,
Cauchy problem,
local well-posedness,
$L^1$ conservation,
blow-up.
\end{abstract}

\maketitle{}

\begin{center}
emails: sanco@brocku.ca, hehuijun@mail2.sysu.edu.cn, zhijun.qiao@utrgv.edu
\end{center}

\section{Introduction}

In the past thirty years, the remarkable Camassa-Holm (CH) equation \cite{CamHol}
 \begin{equation}\label{CHeqn}
m_t+um_x+2u_xm=0,\quad m=u-u_{xx}
\end{equation}
has attracted much attention in the literature on nonlinear systems.
The CH equation can be derived from the Hamiltonian structure of Euler's equations
through an approximation modelling the shallow water scenario.
In particular, it describes the propagation of water waves over a flat bottom in shallow water \cite{CamHol,Joh2002}.
It was also implied as a very special case
in the work of Fuchssteiner and Fokas \cite{FuchFok} on hereditary symmetries.

The CH equation is a completely integrable system possessing
a bi-Hamiltonian structure with an infinite number of conservation laws \cite{CamHol,FuchFok},
and is able to be solved by the inverse scattering method
\cite{ConGerIva2006,ConEsc1998c,Len,ConMcK,BouKosSheTes,Eck}.
One of the main features of interest in the CH equation is
the following peaked wave solution
\begin{equation}
u(x,t)=c\, e^{-|x-ct|},
\quad
c\in \Rnum,
\end{equation}
called a peakon.
These waves are weak solutions which are orbitally stable in the energy space \cite{ConStr,DikMol}.
The CH equation also possesses multi-peakon weak solutions given by a linear superposition of peaked waves having time-dependent amplitudes and speeds,
where the individual peakons retain their shape after interactions.
Another main feature of the CH equation is the phenomena of blow-up
\cite{Con2000a,Con2000b,ConEsc1998b,ConEsc2000,ConEsc1998a,LiOlv}
in which certain initial data for strong solutions evolves such that
the slope of the wave becomes unbounded in a finite time while the wave profile remains bounded.
There are also global weak solutions that persist after the blow up \cite{CamHol,XinZha}.

In addition, the CH equation has several interesting geometrical aspects.
It describes geodesic flows
in the volume-preserving diffeomorphism group of the line (or circle) \cite{Kou,ConKol2003,XinZha,Kol},
as well as non-stretching curve flows in centro-equiaffine planar geometry \cite{QuHanKan}.
Moreover, it possesses algebro-geometric solutions on a symplectic submanifold \cite{Qia2003}.

The CH equation has a close relationship with the Korteweg-de Vries (KdV) equation
$u_t + uu_x + u_{xxx}=0$
which itself is a well-known integrable equation.
Both of these equations have a quadratic nonlinearity and have one Hamiltonian structure in common \cite{AncRec2019,AncChaSzm}.
More importantly, the CH equation is a negative flow in the AKNS hierarchy of integrable equations containing the KdV equation,
and they are connected to each other by a reciprocal transformation.
An interesting related feature is that the bi-Hamiltonian structures of the CH equation and the KdV equation are connected by a tri-Hamiltonian splitting method \cite{OlvRos}.

There are three similarly related pairs of integrable equations that have a cubic nonlinearity.
The best known pair is the modified KdV equation (mKdV)
$u_t + u^2u_x + u_{xxx}=0$ and its peakon counterpart
\begin{equation}\label{mCHeqn}
m_t+ ( m(u^2-u^2_x) ){}_x=0, \quad  m=u-u_{xx},
\end{equation}
called the Fokas-Olver-Rosenau-Qiao (FORQ) equation
which was derived independently in Ref.~\cite{Fok1995,Fuch1996,OlvRos,Qia2006}.
This pair has the same remarkable features as the KdV equation and the CH equation.
For this reason, the FORQ equation is also known as the modified CH equation (mCH).
Its peakon solutions are given by
\begin{equation}
u(x,t)=\pm \sqrt{\tfrac{3}{2}c}\, e^{-|x-ct|},
\quad
c\in \Rnum^+,
\end{equation}
which are unidirectional (in contrast to the CH peakon solutions).
A large amount of work has been done on the mCH equation,
studying the Cauchy problem,
the formation of singularities, wave-breaking mechanism, and peakon stability
(see, e.g., \cite{GuiLiuOlvQu,HimMan} and references therein).

The other two pairs of integrable equations are $U(1)$-invariant extensions of the mKdV--mCH pair
which have been derived recently \cite{AncMob} by the tri-Hamiltonian splitting method.

One pair is given by the complex mKdV equation
$u_t + |u|^2u_x + u_{xxx}=0$
which is also known as the Hirota equation,
and the Hirota-type peakon equation \cite{XiaQia, AncMob}
\begin{align}\label{HPeqn}
m_t+( (|u|^2-|u_x|^2) m){}_x+(\bar u u_x-u \bar u _x)m=0,
\quad
m=u-u_{xx} .
\end{align}
Another pair consists of the NLS equation
$i u_t + |u|^2 u + u_{xx}=0$
and its peakon counterpart \cite{AncMob}
\begin{align}\label{NLSPeqn}
i\,m_t+( (\bar u u_x-u \bar u _x)m ){}_x+(|u|^2-|u_x|^2)m=0,
\quad
m=u-u_{xx} .
\end{align}
Both this NLS-type peakon equation \eqref{NLSPeqn} and the Hirota-type peakon equation \eqref{HPeqn}
are negative flows in the AKNS hierarchy of integrable equations containing the NLS and Hirota equations \cite{AncMob}.
They are also a special case of 2-component peakon systems studied in recent work \cite{XiaQiaZho,XiaQia}
(when the components therein are complexified and $U(1)$-invariance is imposed,
combined with $t\to i t$ in the case of the NLS-type peakon equation).
The real reduction of the Hirota-type peakon equation is given by the FORQ/mCH equation.
In contrast, the NLS-type peakon equation has no real reduction.

Very recently, the NLS-type and Hirota-type peakon equations have been unified into
the following one-parameter family of integrable $U(1)$-invariant peakon equations \cite{AncChaSzm}
\begin{align}\label{unifiedfamilyeqn}
m_t+ (\Re(e ^{i\theta} (u+u_x)(\bar u-\bar u_x)) m ){}_x -i\, \Im(e^{i\theta}(u+u_x)(\bar u-\bar u_x)) \,m=0,
\quad m=u-u_{xx}
\end{align}
where $\theta$ $\in [0,\pi)$ is the parameter.
Note that $\theta=0$ yields the Hirota-type peakon equation \eqref{HPeqn},
and that $\theta=\pi/2$ yields the NLS-type peakon equation \eqref{NLSPeqn}.
A Lax pair and bi-Hamiltonian structure for the family is presented in Ref.~\cite{AncChaSzm}.

A straightforward computation using the method of Ref.~\cite{AncMob}
shows that the peakon weak solutions of the peakon equation family \eqref{unifiedfamilyeqn}
have the form of oscillatory peaked waves
\begin{equation}\label{unifiedfamily-peakon}
u = a e^{i \phi} e^{i \omega t -|x-ct|}
\end{equation}
with the amplitude $a$, speed $c$, and frequency $\omega$ satisfying
\begin{equation}
c = \tfrac{2}{3} a^2 \cos\theta,
\quad
\omega = \tfrac{2}{3} a^2 \sin\theta .
\end{equation}
(Equivalently, $\sqrt{c^2+\omega^2}= \tfrac{2}{3} a^2$ and $\omega/c=\tan\theta$.)
These peakons \eqref{unifiedfamily-peakon}
reduce to the mCH peakon multiplied by a constant phase in the Hirota case $\theta=0$,
whereas in the NLS case $\theta=\tfrac{1}{2}\pi$,
they become a stationary breather.

In the present paper,
we will establish local well-posedness and blow-up results
for strong solutions of the peakon equation family \eqref{unifiedfamilyeqn}.

Local well-posedness will be proved for the Cauchy problem in Besov spaces,
extended to the setting of $U(1)$-invariant norms.
Compared to Sobolev spaces, one advantage of working with Besov spaces is that
the proof of the blow up criterion is insensitive to the regularity index.
The blow-up mechanism will utilize a transport equation method \cite{GuiLiuOlvQu}
similar to how blow-up has been proven for the FORQ/mCH equation.
However,
the FORQ/mCH blow-up analysis relies on having a global sign condition
on the initial data $m_0(x)=m(0,x)$.
It is possible to replace this condition by a local-in-space condition involving
the initial slope of the velocity field $u^2-u_x{}^2$ in addition to the initial data for $m$.
This is made possible by using the transport equation for $m$ to show that the $L^1$ norm $m$ is conserved \cite{CheLiuQuZha}.
The same approach will be used here.
Specifically,
the $L^1$ norm $m$ will be shown to be conserved for initial-value solutions to the peakon equation family \eqref{unifiedfamilyeqn},
and this global conservation will then be used to establish a local-in-space blow up condition
that involves the initial slope of the velocity field $\Re(e ^{i\theta} (u+u_x)(\bar u-\bar u_x))$
and the initial data for $m$.

As we will see in the analysis,
the initial-value solution for the peakon equation family \eqref{unifiedfamilyeqn}
in general has oscillatory-in-time properties along the characteristic of the velocity field,
whereas for the FORQ/mCH equation \eqref{mCHeqn} it does not.
Consequently,
any global-in-space sign conditions on initial data for $m$ are not a priori preserved in time.
This interesting feature may indicate some essential difference between
the dynamics in the $U(1)$-invariant peakon equation compared to the FORQ/mCH equation,
since the $U(1)$-invariant peakon equation is intrinsically a two-component coupled system.
The oscillatory-in-time dynamics has not been seen previously in other two-component peakon systems.

The rest of the paper is organized as follows.

In Section~\ref{sec:mainresults},
the main results are stated.
Proofs of the local well-posedness Theorem~\ref{thm:wellposedness}
and the $L^1$ conservation Theorem~\ref{thm:L1conservation}
for the peakon equation \eqref{unifiedfamilyeqn}
are provided in Section~\ref{sec:wellposedness} and Section~\ref{sec:L1conservation}, respectively.
In Sections~\ref{sec:blowupcriteria} and~\ref{sec:blowup},
the precise blow-up criterion Theorem~\ref{thm:criterion}
and the main blow-up result Theorem~\ref{thm:blowup}
are presented, respectively.
Finally,
some concluding remarks are made in Section~\ref{sec:conclude}.

The basic aspects of Besov spaces
and the corresponding linear transport theory
needed for the main results
are summarized in the Appendix. 

\section{Main Results}\label{sec:mainresults}

We begin by stating some preliminaries that will be needed for the main results.

It will be convenient to introduce the expression
\begin{equation}\label{Q}
Q=(u+u_x)(\bar u-\bar u_x)= |u|^2-|u_x|^2+2i\Im (\bar u u_x)
\end{equation}
so that the peakon equation family \eqref{unifiedfamilyeqn} can be written succinctly as
\begin{equation}\label{m-eqn}
m_t+( \Re(e^{i\theta}Q) m ){}_x -i\, \Im(e^{i\theta}Q) m=0,
\quad
m=u-u_{xx} .
\end{equation}
The transport form of this equation is given by
\begin{equation}\label{m.transporteqn}
m_t+ J m_x = Km
\end{equation}
with
\begin{gather}
J = \Re(e^{i\theta}Q),
\label{J}
\\
K = i \Im(e^{i\theta}Q)  -\Re(e^{i\theta}Q_x),
\label{K}
\end{gather}
where
\begin{equation}\label{Qx}
Q_x = (u+u_x)\bar m - (\bar u -\bar u_x) m = \Re(\bar u_x m) -2i\Im (\bar u m) .
\end{equation}

Since equation \eqref{m-eqn} is invariant under the $U(1)$ group of constant phase rotations
$$
m \to e^{i\phi} m,
\quad
\phi \in \Rnum,
$$
it will be natural to work with $U(1)$-invariant norms.
For a complex function $f=f_1 + i f_2$ on $\Rnum$,
we define
\begin{equation}\label{U1inv-Lpnorm}
\|f\|_{L^p(\Rnum)}= \left(\int_{\Rnum}|f_1(x)+i f_2(x)|^p\d  x\right)^{\frac{1}{p}}
=\left(\int_{\Rnum}(|f_1|^2+|f_2|^2)^{\frac p2}\d  x\right)^{\frac{1}{p}}
=\left(\int_{\Rnum}(f\bar f)^{\frac p2}\d  x\right)^{\frac{1}{p}}
\end{equation}
which is manifestly invariant under $f\to e^{i\phi}f$.
This $U(1)$-invariant norm for complex functions is equivalent to the standard norm
$$
\|(f_1,f_2)\|_{L^p(\Rnum)} = \|f_1\|_{L^p(\Rnum)}+\|f_2\|_{L^p(\Rnum)}
$$
for corresponding pairs of real functions,
as shown in the Appendix.
Consequently, the Besov norms
\begin{equation}\label{U1inv-besovnorm}
\|m\|_{B^s_{p,r}(\Rnum)} = \|m_1 + im_2\|_{B^s_{p,r}(\Rnum)}
\end{equation}
and
$$
\|(m_1,m_2)\|_{B^s_{p,r}(\Rnum)} = \|m_1\|_{B^s_{p,r}}+\|m_2\|_{B^s_{p,r}(\Rnum)}
$$
on the complex peakon momentum variable $m=m_1+im_2$ are equivalent.
Moreover, for any constant $C\in\Cnum$,
$$
\|Cm\|_{B^s_{p,r}(\Rnum)} = |C| \|m\|_{B^s_{p,r}(\Rnum)} .
$$

We will now state the main theorems.
Hereafter, all norms will refer to the $U(1)$-invariant versions \eqref{U1inv-Lpnorm} and \eqref{U1inv-besovnorm}.

We begin with the statement of local well-posedness in Besov spaces.

\begin{theo}\label{thm:wellposedness}
\thmname{Well-posedness}
Let $1\leq p,r\leq\infty$ and $s>\max(\frac12,\frac{1}{p})$.
Then for any initial data $m_0\in B^s_{p,r}(\Rnum)$,
the peakon equation \eqref{unifiedfamilyeqn} has a unique solution
\begin{equation}\label{wellposedness}
m\in \begin{cases}
\C([0,T^*); B^s_{p,r}(\Rnum))\cap \C^1([0,T^*); B^{s-1}_{p,r}(\Rnum)),
& \text{ if } r<\infty
\\
\C_\omega([0,T^*); B^s_{p,\infty}(\Rnum))\cap \C^{0,1}([0,T^*); B^{s-1}_{p,\infty}(\Rnum)),
& \text{ if }  r=\infty
\end{cases}
\end{equation}
for $0\leq t < T^*$,
where the maximal existence time satisfies
\begin{equation}
T^*\geq \frac{C}{\|m_0\|_{B^s_{p,r}(\Rnum)}^2}
\end{equation}
for some constant $C>0$.
Moreover, the solution $m$ depends continuously on the initial data $m_0$.
\end{theo}

Next we state the $L^1$ conservative law.

\begin{theo}\label{thm:L1conservation}
\thmname{$L^1$ Conservation}
For initial data $m_0\in H^s(\Rnum)\cap L^1(\Rnum)$, with $s>\frac12$,
let $T^*$ be the maximal existence time of the corresponding solution $m$
to the peakon equation \eqref{unifiedfamilyeqn}.
Then, for $0\leq t < T^*$,
\begin{equation}\label{L1.conservation}
\frac{\d}{\d t}\int_{\Rnum}|m(t,x)|\d x = 0 ,
\end{equation}
and hence $\|m\|_{L^1(\Rnum)} =\|m_0\|_{L^1(\Rnum)}$.
\end{theo}

In addition to this conservation law,
both of the Hamiltonians in the bi-Hamiltonian structure \cite{AncChaSzm}
for strong solutions are conserved:
\begin{gather*}
\frac{\d}{\d t}\int_{\Rnum} \big( \sin\theta\Re(\bar u m) +\cos\theta\Im(\bar u_x m) \big)\,dx =0,
\\
\begin{aligned}
\frac{\d}{\d t}\int_{\Rnum} \big( &
\tfrac{1}{4}(|u|^2 -|u_x|^2)( \sin\theta\Re(\bar{u}m) + \cos\theta \Im(\bar{u}_xm) )
\\&
+ \tfrac{1}{2}\Im(\bar{u}u_x) ( \sin\theta\Im(\bar{u}_xm) -\cos\theta \Re(\bar{u}m)  )
\big)\,dx =0 .
\end{aligned}
\end{gather*}
The first conservation law reduces to conservation of the $U(1)$-invariant $H^1$ norm of $u$
when $\theta=\tfrac{1}{2}\pi$.

Now we state the precise blow-up criterion for the peakon equation \eqref{unifiedfamilyeqn},
which comes from transport theory.

\begin{theo}\label{thm:criterion}
\thmname{Blow-up criterion}
For initial data $m_0\in H^s(\Rnum)$, with $s>\frac12$,
let $T>0$ be the maximal existence time of the corresponding solution $m$ to the peakon equation \eqref{unifiedfamilyeqn}.
Then $m$ blows up in finite time if and only if
\begin{equation}\label{blowup.criterion}
\limsup\limits_{t\uparrow T}\inf\limits_{x\in\Rnum} \Re( e^{i\theta} ((u+u_x)\bar m - (\bar u -\bar u_x) m) )
=-\infty .
\end{equation}
\end{theo}

Combining this blow-up criterion and the $L^1$ conservation law,
we can obtain the following blow-up result.

\begin{theo}\label{thm:blowup}
\thmname{Blow-up}
Let $m_0\in H^s(\Rnum)\cap L^1(\Rnum)$, with $s>\frac12$,
such that there exists some $x_0\in \Rnum$ and some constant $C_0>0$
for which
\begin{equation}\label{blowup.condition}
\Re( e^{i\theta}Q_x(x_0)) \leq -\sqrt{2C_0|m_0(x_0)|}<0
\end{equation}
with $Q_x$ given by expression \eqref{Qx}.
Then the corresponding solution $m$ to the peakon equation \eqref{unifiedfamilyeqn}
blows up at the finite time
\begin{equation}\label{blowup.time}
T^*=\frac{|\Re( e^{i\theta}Q_x )(x_0)|  -\sqrt{|\Re( e^{i\theta}Q_x(x_0))|^2 -2C_0|m_0(x_0)|}}{C_0|m_0(x_0)|} > 0 .
\end{equation}
\end{theo}


We will prove Theorems~\ref{thm:wellposedness}, \ref{thm:L1conservation}, \ref{thm:criterion}, and~\ref{thm:blowup}
in sections~\ref{sec:wellposedness}, \ref{sec:L1conservation}, \ref{sec:blowupcriteria}, and~\ref{sec:blowup},
respectively.

\section{Local well-posedness}\label{sec:wellposedness}

The local well-posedness of the peakon equation \eqref{unifiedfamilyeqn}
in Theorem~\ref{thm:wellposedness}
can be established by means of
a standard Picard scheme using Littlewood-Paley decomposition theory in Besov spaces
(see the Appendix for details),
applied to the equation in transport form
\begin{align}\label{ivp-transport}
m_t+Jm_x=Km,
\quad
m|_{t=0}=m_0 .
\end{align}

There are three main steps, which are proved in the subsequent three subsections.
Hereafter, all spaces will be understood to have the domain $\Rnum$,
and we assume $s>\max\{\frac{1}{2},\frac{1}{p}\}$.

To establish uniform boundedness of the approximate solutions in the Picard scheme,
we work in the space $B^s_{p,r}$.
Convergence of these solutions is then proved in the space $B^{s-1}_{p,r}$,
which is a weaker space than $B^{s}_{p,r}$
because of the embedding relation $B^{s}_{p,r}\hookrightarrow B^{s-1}_{p,r}$.

\subsection{Uniform boundedness of the approximate solutions}

Consider the linear transport system
\begin{equation}\label{linear.transport.sys}
m^{(k+1)}_t+J^{(k)}m^{(k+1)}_x =K^{(k)} m^{(k)},
\quad
m^{(k+1)}|_{t=0}=S_{k+1}m_0,
\quad
k\in\Nnum
\end{equation}
with smooth data $m_0$,
where $S_{q}$ is the low frequency cut-off operator (see Remark~\ref{rem:S.op}).
We start from $u^{(0)}=m^{(0)}=0$, and assume that, for all $0\leq t <T$,
\begin{equation}\label{mk.wellposedness.spaces}
m^{(k)}(t,x)\in E^s_{p,r}:= \begin{cases}
\C([0,T); B^s_{p,r})\cap \C^1([0,T); B^{s-1}_{p,r}),
& \text{ if } r<\infty
\\
\C_\omega([0,T); B^s_{p,\infty})\cap \C^{0,1}([0,T); B^{s-1}_{p,\infty}),
& \text{ if }  r=\infty
\end{cases}.
\end{equation}
Our goal in this subsection will be to show that \eqref{mk.wellposedness.spaces} holds for $m^{(k+1)}$.

From linear transport theory, the solution $m^{(k+1)}$ of equation \eqref{linear.transport.sys}
satisfies the a priori estimate
(see Lemmas~\ref{apriori-estimates} and~\ref{Luowei})
\begin{equation}\label{mk+1.estimate}
\|m^{(k+1)}\|_{B^{s}_{p,r}}
\leq
e^{C\alpha^{(k)}(t)} \|S_{k+1}m_0\|_{B^s_{p,r}}+\int_0^t e^{C(\alpha^{(k)}(t)-\alpha^{(k)}(\tau))} \|K^{(k)}(\tau)m^{(k)}(\tau)\|_{B^s_{p,r}} \d\tau
\end{equation}
where $C$ is some positive constant,
and where
\begin{align*}\label{A-expression}
\alpha^{(k)}(t):=\begin{cases}
\int_0^t\|\partial_x J^{(k)}(\tau)\|_{B^{1/p}_{p,\infty}\cap L^\infty}\,\d\tau,
& s<1+\frac{1}{p},
\\
\int_0^t\|\partial_x J^{(k)}(\tau)\|_{B^{s-1}_{p,r}}\,\d\tau,
& s>1+\frac{1}{p},~r>1;
s=1+\frac{1}{p},~ r=1
\\
\int_0^t\|J^{(k)}(\tau)\|_{B^{s+1}_{p,r}}\,\d\tau,
& s=1+\frac{1}{p},~r>1
\end{cases}.
\end{align*}

Each term on the right-hand side of \eqref{mk+1.estimate}
can be controlled in terms of $\|m^{(k)}\|_{B^{s}_{p,r}}$ and $\|m_0\|_{B^{s}_{p,r}}$ as follows.

First,
we have
\begin{equation*}
\|S_{k+1}m_0\|_{B^{s}_{p,r}}\leq C_0\|m_0\|_{B^{s}_{p,r}}
\end{equation*}
(see Remark~\ref{rem:S.op})
for some positive constant $C_0$.

Next,
since $u^{(k)}=(1-\partial_x^2)^{-1}m^{(k)}$
where the symbol of the operator $(1-\partial_x^2)^{-1}$ is
$(1+\xi^2)^{-1}$ which is an $S^{-2}$-multiplier,
then by Lemma~\ref{lem:Besov.spaces}(viii)) we have
\begin{align*}
\|u^{(k)}\|_{B^{s}_{p,r}} & =
\|(1-\partial_x^2)^{-1} m^{(k)}\|_{B^{s}_{p,r}}\leq C\|m^{(k)}\|{B^{s-2}_{p,r}} \leq C\|m^{(k)}\|_{B^{s}_{p,r}}
\\
\|u^{(k)}_x\|_{B^{s}_{p,r}} & =
\|(1-\partial_x^2)^{-1} m^{(k)}_x\|_{B^{s}_{p,r}}\leq C\|m^{(k)}\|_{B^{s-1}_{p,r}}\leq C\|m^{(k)}\|_{B^{s}_{p,r}} .
\end{align*}
Hence we obtain
\begin{equation*}
\|u^{(k)}\pm u^{(k)}_x\|_{B^{s}_{p,r}}
\leq C\|m^{(k)}\|_{B^{s}_{p,r}} .
\end{equation*}
To proceed,
we use the properties that $B^s_{p,r}$ is an algebra and obeys the embedding
$B^s_{p,r}\hookrightarrow B^{s-1}_{p,r}\hookrightarrow B^{s-2}_{p,r}$,
whereby
\begin{equation}\label{Jx.ReK.estimate}
\|\Re K\|_{B^{s}_{p,r}}
=\|J_x^{(k)}\|_{B^{s}_{p,r}}
\leq \|Q_x^{(k)}\|_{B^{s}_{p,r}}
\leq C_1\big\| Q^{(k)}\big\|_{B^{s+1}_{p,r}}
\leq 4C_1\|m^{(k)}\|_{B^s_{p,r}}^2
\end{equation}
and
\begin{equation}\label{ImK.estimate}
\|\Im K\|_{B^{s}_{p,r}}
\leq \|Q^{(k)}\|_{B^{s}_{p,r}}
\leq 4\|m^{(k)}\|_{B^{s}_{p,r}}^2
\end{equation}
where $C_1$ is some positive constant.
Then we have the bounds
\begin{align*}
C\alpha^{(k)}(t) & \leq C_2\int_0^t \|m^{(k)}(\tau)\|_{B^s_{p,r}}^2 \,\d\tau,
\\
C(\alpha^{(k)}(t)-\alpha^{(k)}(\tau)) & \leq C_2\int_\tau^t \|m^{(k)}(\tau)\|_{B^s_{p,r}}^2\,\d \tau,
\end{align*}
where $C_2=4C_1C$ is a positive constant.

Hence, for the two terms on the right-hand side of \eqref{mk+1.estimate},
we obtain
\begin{equation*}
e^{C\alpha^{(k)}(t)} \|S_{k+1}m_0\|_{B^s_{p,r}}
\leq
C_0 e^{ C_2\int_0^t \|m^{(k)}(\tau)\|_{B^s_{p,r}}^2 \,\d\tau} \|m_0\|_{B^{s}_{p,r}}
\end{equation*}
and
\begin{equation*}
\int_0^t e^{C(\alpha^{(k)}(t)-\alpha^{(k)}(\tau))} \|K^{(k)}(\tau)m^{(k)}(\tau)\|_{B^s_{p,r}}\,\d\tau
\leq
4(1+C_1) \int_0^t e^{ C_2\int_\tau^t \|m^{(k)}(t')\|_{B^s_{p,r}}^2\,\d t' } \|m^{(k)}(\tau)\|_{B^{s}_{p,r}}^3\,\d\tau .
\end{equation*}
Combining these terms, we get the estimate
\begin{equation}\label{mk+1.estimate2}
\|m^{(k+1)}\|_{B^{s}_{p,r}}
\leq
C_0 e^{ C_2\int_0^t \|m^{(k)}(t')\|_{B^s_{p,r}}^2 \,\d t' } \|m_0\|_{B^{s}_{p,r}}
+ 4(1+C_1) \int_0^t e^{ C_2\int_\tau^t \|m^{(k)}(t')\|_{B^s_{p,r}}^2\,\d t' } \|m^{(k)}(\tau)\|_{B^{s}_{p,r}}^3\,\d\tau .
\end{equation}

Now we want to control the terms on the right-hand side of \eqref{mk+1.estimate2}
strictly in terms of $\|m_0\|_{B^{s}_{p,r}}$.
We will use the an induction argument.

Fix a $T>0$ such that
\begin{align}\label{T.condition}
1-\c \|m_0\|_{B^s_{p,r}}^2 T>0
\end{align}
for some $\c>0$.
Suppose that, for a.e.~$t\in [0,T]$ and $k\in \mathbb N$,
$m^{(k)}(t)$ obeys the bound
\begin{equation}\label{mk.bound}
\|m^{(k)}(t)\|_{B^{s}_{p,r}}
\leq \frac{\cc \|m_0\|_{B^{s}_{p,r}}}{\sqrt{1-\c \|m_0\|_{B^s_{p,r}}^2 t}}
\end{equation}
where $\cc>0$ will be suitable chosen.
Clearly, this bound is valid when $k=0$, since $m^{(0)}=0$.
To show that the bound \eqref{mk.bound} holds for $m^{(k+1)}(t)$,
we first note that
\begin{equation*}
\int_\tau^t \|m^{(k)}(t')\|_{B^s_{p,r}}^2 \,\d t'
\leq \frac{\cc^2}{\c}\ln\left(\frac{1-\c \|m_0\|_{B^s_{p,r}}^2 \tau}{1-\c \|m_0\|_{B^s_{p,r}}^2 t}\right) ,
\end{equation*}
and thus
\begin{align*}
e^{C_2\int_\tau^t \|m^{(k)}(t')\|_{B^s_{p,r}}^2\d t'}
\leq \left(\frac{1-\c \|m_0\|_{B^s_{p,r}}^2\tau}{1-\c \|m_0\|_{B^s_{p,r}}^2 t}\right)^{a}
\end{align*}
where $a=C_2 \cc^2/\c$.
Then we use the estimate \eqref{mk+1.estimate2} to obtain
\begin{align*}
\|m^{(k+1)}\|_{B^{s}_{p,r}}
& \leq
(1-\c \|m_0\|_{B^s_{p,r}}^2 t)^{-a} \Big(
C_0 \|m_0\|_{B^s_{p,r}}
+ 4(1+C_1) \cc^3 \|m_0\|_{B^s_{p,r}}^3
\int_0^t (1-\c \|m_0\|_{B^s_{p,r}}^2 \tau)^{a -3/2}\,\d \tau
\Big)
\\
& =
\big( C_0 -\tfrac{8}{1-2a} (1+C_1) (\cc^3/\c)  \big) \|m_0\|_{B^s_{p,r}} (1-\c \|m_0\|_{B^s_{p,r}}^2 t)^{-a}
\\&\qquad
+\tfrac{8}{1-2a} (1+C_1) (\cc^3/\c)  \|m_0\|_{B^s_{p,r}} (1-\c \|m_0\|_{B^s_{p,r}}^2 t)^{-1/2}
\\
& =
\tfrac{8}{1-2a}(1+C_1) (\cc^3/\c)\frac{\|m_0\||_{B^{s}_{p,r}}}{\sqrt{1-\c \|m_0\|_{B^s_{p,r}}^2 t}}
\end{align*}
if $(1+C_1)(\cc^3/\c) = \tfrac{1-2a}{8}C_0 >0$.
This will establish the bound \eqref{mk.bound} for $m^{(k+1)}(t)$
if $(1+C_1) (\cc^3/\c)\leq \tfrac{1-2a}{8}\cc$.
Hence, we need the positive constants $\cc$ and $\c$ to satisfy
$C_2 \cc^2 <\tfrac{1}{2}\c$,
$\cc\geq C_0$,
and
$\cc^2(1+C_1)\leq \tfrac{1}{8}(\c -2C_2 \cc^2)$.
It is sufficient to let
\begin{equation*}
C=\max(4(1+C_1),C_2,C_0)
\end{equation*}
and choose $\c=4 C^3$ and $\cc=C$,
which can be readily verified to satisfy the preceding three inequalities.

The induction step is now proved,
and therefore, $\{m^{(k)}(t)\}_{k\in \Nnum}$ in $L^\infty([0,T];B^{s}_{p,r})$
is uniformly bounded by the positive constant
\begin{equation}\label{mk.bound.CT}
C_T := \frac{C \|m_0\|_{B^{s}_{p,r}}}{\sqrt{1- 4C^3 \|m_0\|_{B^s_{p,r}}^2 T}} .
\end{equation}
By Lemma \ref{regularity}, we see that the sequence $\{m^{k}\}_{k\in\mathbb N}$ is bounded in $E^s_{p,r}(T).$

\subsection{Convergence and regularity of the solutions}

Our next goal will be to prove that $\{m^{(k)}\}_{k\in \Nnum}$ is a Cauchy sequence in $\mathcal C([0,T];B^{s-1}_{p,r})$.
Consider
\begin{equation*}
m^{(k+1,l)}:=m^{(k+l+1)}-m^{(k)}
\end{equation*}
which, from the linear transport system \eqref{linear.transport.sys},
satisfies
\begin{equation*}
m^{(k+1,l)}_t +J^{(k+l)} m^{(k+1,l)}_x
= R^{(k+1,l)},
\quad
R^{(k+1,l)}: = (J^{(k)}-J^{(k+l)}) m_x^{(k+1)} + K^{(k+l)}m^{(k+l)} - K^{(k)}m^{(k)},
\end{equation*}
with initial data
\begin{equation*}
m^{(k+1,l)}|_{t=0} =(S_{k+l+1}-S_{k+1})m_0 .
\end{equation*}
Applying the a priori estimate given by Lemmas~\ref{apriori-estimates} and~\ref{Luowei},
and recalling that $s>\max(\frac{1}{2},\frac{1}{p})$,
we obtain
\begin{equation}\label{summary1}
\begin{aligned}
\|m^{(k+1,l)}(t)\|_{B^{s-1}_{p,r}}
& \leq e^{C \beta^{(k+l)}(t)}\|(S_{k+l+1}-S_{k+1})m_0\|_{B^{s-1}_{p,r}}
\\&\qquad
+ C\int_0^t e^{C(\beta^{(k+l)}(t)-\beta^{(k+l)}(\tau))} \| R^{(k+1,l)}(\tau) \|_{B^{s-1}_{p,r}}\,\d \tau
\end{aligned}
\end{equation}
where
\begin{align}\label{N-expression}
\beta^{(k)}(t):=\begin{cases}
\int_0^t\|\partial_x J^{(k)}(\tau)\|_{B^{1/p}_{p,\infty}\cap L^\infty}\,\d\tau,
& s-1<1+\frac{1}{p},
\\
\int_0^t\|\partial_x J^{(k)}(\tau)\|_{B^{s-2}_{p,r}}\,\d\tau,
& s-1>1+\frac{1}{p},~r>1;
s-1=1+\frac{1}{p},~ r=1
\\
\int_0^t\|J^{(k)}(\tau)\|_{B^{s}_{p,r}}\,\d\tau,
& s-1=1+\frac{1}{p},~r>1
\end{cases}
\end{align}
It is straightforward to bound $\beta^{(k)}(t)$, for all $t\in[0, T]$,
due to the boundedness of $\|m^{(k)}\|_{B^{s}_{p,r}}$.
Specifically,
in the first two cases,
\begin{align*}
& \|\partial_xJ^{(k)}\|_{B^{s-2}_{p,r}}\leq C\|J^{(k)}\|_{B^{s-1}_{p,r}}\leq C\|J^{(k)}\|_{B^{s}_{p,r}}
=C\|e^{i\theta}(u^{(k)}+u_x^{(k)})(\bar u^{(k)}+\bar u_x^{(k)})\|_{B^{s}_{p,r}}\leq C
\|m^{(k)}\|^2_{B^{s}_{p,r}} ,
\\
& \|J^{(k)}\|_{B^{s}_{p,r}}
=\|e^{i\theta}(u^{(k)}+u_x^{(k)})(\bar u^{(k)}+\bar u_x^{(k)})\|_{B^{s}_{p,r}}\leq C
\|m^{(k)}\|^2_{B^{s}_{p,r}} .
\end{align*}
In the third case,
since $s>1/p$ whereby $s+1>1+1/p$,
we have
$B^s_{p,r}\hookrightarrow L^\infty$ and $B^{s+1}_{p,r}\hookrightarrow B^{1+1/p}_{p,\infty}$
so thus
$\|\partial_xJ\|_{L^\infty}\leq C\|\partial_xJ\|_{B^s_{p,r}}\leq C\|J\|_{B^{s+1}_{p,r}}\leq C\|m\|^2_{B^s_{p,r}}$
and
$\|\partial_xJ\|_{B^{1/p}_{p,\infty}}\leq C\|J\|_{B^{1/p+1}_{p,\infty}}\leq C\|J\|_{B^{s+1}_{p,r}}\leq C\|m\|^2_{B^s_{p,r}}$.
This implies
\begin{align*}
\|\partial_xJ\|_{B^{1/p}_{p,\infty}\cap L^\infty}=
\|\partial_xJ\|_{B^{1/p}_{p,\infty}}+\|\partial_xJ\|_{L^\infty}\leq C\|m\|^2_{B^s_{p,r}} .
\end{align*}
Hence,
\begin{align}\label{plugging1}
\beta^{(k)}(t)\leq C\int_0^t\|m^{(k)}(\tau)\|_{B^s_{p,r}}^2\,\d\tau\leq C_T.
\end{align}

To control the term in the a priori estimate \eqref{summary1}
involving the initial data $m_0$,
we make the following direct computation.
Firstly, we have
\begin{equation}
\|(S_{k+l+1}-S_{k+1})m_0\|_{B^{s-1}_{p,r}}
=\big\|\sum\limits_{1\leq i\leq l} \Delta_{k+i} m_0\big\|_{B^{ s-1}_{p,r}}
=\big(\sum\limits_{j\geq-1} 2^{jr( s-1)}\|\Delta_j (\sum\limits_{1\leq i\leq l} \Delta_{k+i} m_0)\|^r_{L^p} \big)^{\frac{1}{r}} .
\end{equation}
Secondly, the sum over $j\geq -1$ has only a finite number of non-zero terms,
since
$\Delta_j\sum\limits_{1\leq i\leq l}\Delta_{k+i} f=0$
for $j\leq k-1$ and $j\geq k+l+2$
by Remark~\ref{rem:S.op}(ii),
where $f$ is any tempered distributions in $\mathcal S'$.
Therefore, we obtain
\begin{equation}
\begin{aligned}
\|(S_{k+l+1}-S_{k+1})m_0\|_{B^{s-1}_{p,r}}
&=
\big(\sum\limits_{j=k}^{k+l+1} 2^{-jr}2^{jrs}\|\Delta_j(\sum\limits_{1\leq i\leq l} \Delta_{k+i} m_0)\|^r_{L^p} \big)^{\frac{1}{r}}\\
&=
\big(\sum\limits_{j=k}^{k+l+1} 2^{-jr}2^{jrs}\|\Delta_j(S_{k+l+1}-S_{k+1})m_0\|^r_{L^p} \big)^{\frac{1}{r}}\\
&\leq
\big(\sum\limits_{j=k}^{k+l+1} 2^{-k}2^{jrs}\|\Delta_j(S_{k+l+1}-S_{k+1})m_0\|^r_{L^p} \big)^{\frac{1}{r}} .
\end{aligned}
\end{equation}
Thirdly, we use the inequalities
\begin{align*}
\|\Delta_j(S_{k+l+1}-S_{k+1})m_0\|_{L^p}
&=\|(S_{k+l+1}-S_{k+1})\Delta_jm_0\|_{L^p}\\
&\leq \|S_{k+l+1}\Delta_jm_0\|_{L^p} +\|S_{k+1}\Delta_jm_0\|_{L^p}\\
&\leq C\|\Delta_j m_0\|_{L^p} .
\end{align*}
This yields
\begin{equation}\label{plugging2}
\begin{aligned}
\|(S_{k+l+1}-S_{k+1})m_0\|_{B^{s-1}_{p,r}}
&\leq
C\, 2^{-k}\big(\sum\limits_{j=k}^{k+l+1} 2^{jsr}\|\Delta_jm_0\|^r_{L^p} \big)^{\frac{1}{r}}\\
&\leq
C\, 2^{-k}\big(\sum\limits_{j\geq -1} 2^{jsr}\|\Delta_jm_0\|^r_{L^p} \big)^{\frac{1}{r}}\\
&=
C\,2^{-k}\|m_0\|_{B^s_{p,r}} ,
\end{aligned}
\end{equation}
which controls the initial data term in the a priori estimate \eqref{summary1}.

Next, we need to control the term $\| R^{(k+1,l)}\|_{B^{s-1}_{p,r}}$ in the a priori estimate \eqref{summary1}.
It consists of three separate parts:
\begin{align*}
\begin{aligned}
\| R^{(k+1,l)}\|_{B^{s-1}_{p,r}}
& \leq
\| (J^{(k)}-J^{(k+l)}) m_x^{(k+1)} \|_{B^{s-1}_{p,r}}
+ \| K^{(k)}m^{(k)} - K^{(k+l)}m^{(k+l)} \|_{B^{s-1}_{p,r}} \\
& \leq
\| (Q^{(k)}-Q^{(k+l)}) m_x^{(k+1)} \|_{B^{s-1}_{p,r}}
+ \| Q^{(k)}m^{(k)} - Q^{(k+l)}m^{(k+l)} \|_{B^{s-1}_{p,r}}
\\&\qquad
+ \| Q_x^{(k)}m^{(k)} - Q_x^{(k+l)}m^{(k+l)} \|_{B^{s-1}_{p,r}}
\end{aligned}
\end{align*}
where
\begin{gather}
(Q^{(k)} -Q^{(k+l)}) m^{(k+1)}_x =
( (u^{(k)}+u^{(k)}_x)(\bar u^{(k)}-\bar u^{(k)}_x) - (u^{(k+l)}+u^{(k+l)}_x)(\bar u^{(k+l)}-\bar u^{(k+l)}_x)
)m^{(k+1)}_x,
\label{Jterms}
\\
\begin{aligned}
Q^{(k)}m^{(k)} - Q^{(k+l)}m^{(k+l)} & =
\big( (u^{(k)}+u^{(k)})(\bar u^{(k)}-\bar u_x^{(k)}) \big) m^{(k)}
\\&\qquad
-\big( (u^{(k+l)}+u^{(k+l)})(\bar u^{(k+l)}-\bar u_x^{(k+l)}) \big) m^{(k+l)},
\label{Kterms1}
\end{aligned}
\\
\begin{aligned}
Q^{(k)}_x m^{(k)} - Q^{(k+l)}_x m^{(k+l)} & =
\big( (u^{(k)}+u^{(k)})(\bar u^{(k)}-\bar u_x^{(k)}) \big)_x m^{(k)}
\\&\qquad
-\big( (u^{(k+l)}+u^{(k+l)})(\bar u^{(k+l)}-\bar u_x^{(k+l)}) \big)_x m^{(k+l)}.
\label{Kterms2}
\end{aligned}
\end{gather}
Note that $B^{s-1}_{p,r}$ is not guaranteed to be a Banach algebra,
due to $s>\max(\frac{1}{2},\frac{1}{p})$.
But we still have a Moser-type inequality
$\|fg\|_{B^{s-1}_{p,r}}\leq C\|f\|_{B^{s-1}_{p,r}}\|g\|_{B^{s}_{p,r}}$
provided by Lemma~\ref{Moser}(ii)
which holds under the conditions
$s_1 \leq s_2$, $s_2 \geq \frac{1}{p}$, and $s_1 + s_2 > \max(0, \frac{2}{p}- 1)$,
satisfied by the regularity indices $s_1=s-1$ and $s_2=s$.
We can then control each part \eqref{Jterms}--\eqref{Kterms2}
by the method of adding and subtracting terms to produce factors of the form
$\| u^{(k+l)} - u^{(k)} \|_{B^{s}_{p,r}} \leq \|m^{(k+1,l)}\|_{B^{s-1}_{p,r}}$
or $\| u_x^{(k+l)} - u_x^{(k)} \|_{B^{s}_{p,r}} \leq \|m^{(k+1,l)}\|_{B^{s-1}_{p,r}}$.

In expression \eqref{Jterms}, consider the group of terms
\begin{equation*}
( \bar u^{(k+1)}_x u^{(k+1)}_x -\bar u^{(k)}_x u^{(k)}_x )m^{(k+1)}_x
=( \bar u^{(k+1)}_x (u^{(k+1)}_x-u_x^{(k)}) +u^{(k)}_x(\bar u^{(k+1)}_x-\bar u^{(k)}_x) )m^{(k+1)}_x .
\end{equation*}
Hence we have
\begin{equation*}
\begin{aligned}
\| ( \bar u^{(k+1)}_x u^{(k+1)}_x -\bar u^{(k)}_x u^{(k)}_x )m^{(k+1)}_x \|_{B^{s-1}_{p,r}}
& = \| ( \bar u^{(k+l)}_x (u^{(k+l)}_x-u_x^{(k)}) +u^{(k)}_x(\bar u^{(k+l)}_x-\bar u^{(k)}_x)  )m^{(k+1)}_x \|_{B^{s-1}_{p,r}} \\
&\leq C \|u_x^{(k+l)}-u_x^{(k)}\|_{B^{s}_{p,r}} \big(\|u_x^{(k+l)}\|_{B^{s}_{p,r}} +\|u_x^{(k)}\|_{B^{s}_{p,r}}\big) \|m^{(k+1)}_x\|_{B^{s-1}_{p,r}} \\
&\leq C \|m^{(k+l,k)}\|_{B^{s-1}_{p,r}} \big(\|m^{(k+1)}\|_{B^{s}_{p,r}} +\|m^{(k)}\|_{B^{s}_{p,r}}\big) \|m^{(k+1)}\|_{B^{s}_{p,r}} \\
&\leq C \|m^{(k+l,k)}\|_{B^{s-1}_{p,r}} \big(\|m^{(k+1)}\|^2_{B^{s}_{p,r}} +\|m^{(k)}\|^2_{B^{s}_{p,r}}\big) .
\end{aligned}
\end{equation*}
The other three groups of terms in expression \eqref{Jterms} can be estimated in the same way.

This gives
\begin{equation}\label{Jterms.estimate}
\| (Q^{(k)}-Q^{(k+l)}) m_x^{(k+1)} \|_{B^{s-1}_{p,r}}
\leq C \|m^{(k+l,k)}\|_{B^{s-1}_{p,r}} \big(\|m^{(k+1)}\|^2_{B^{s}_{p,r}} +\|m^{(k)}\|^2_{B^{s}_{p,r}}\big) .
\end{equation}
Similarly, we can obtain the same estimate for the terms in \eqref{Kterms1} and \eqref{Kterms2}:
\begin{equation}\label{Kterms.estimate}
\begin{aligned}
& \| Q_x^{(k)}m^{(k)} - Q_x^{(k+l)}m^{(k+l)} \|_{B^{s-1}_{p,r}}
+\| Q^{(k)}m^{(k)} - Q^{(k+l)}m^{(k+l)} \|_{B^{s-1}_{p,r}}
\\
& \leq
C \|m^{(k+l,k)}\|_{B^{s-1}_{p,r}} \big(\|m^{(k+1)}\|^2_{B^{s}_{p,r}} +\|m^{(k)}\|^2_{B^{s}_{p,r}}\big) .
\end{aligned}
\end{equation}
Combining these estimates \eqref{Jterms.estimate} and \eqref{Kterms.estimate},
we have
\begin{equation}\label{plugging3}
\begin{aligned}
\| R^{(k+1,l)}\|_{B^{s-1}_{p,r}}
& \leq C \|m^{(k+l,k)}\|_{B^{s-1}_{p,r}} \big(\|m^{(k+1)}\|^2_{B^{s}_{p,r}} +\|m^{(k)}\|^2_{B^{s}_{p,r}}\big) \\
& \leq C_T \|m^{(k+l,k)}\|_{B^{s-1}_{p,r}}.
\end{aligned}
\end{equation}

Now we substitute the main estimates \eqref{plugging2} and \eqref{plugging3}
into the a priori estimate \eqref{summary1}
and use the estimate \eqref{plugging1} for $\beta^{(k)}(t)$.
This yields
\begin{align*}
\|m^{(k+1,l)}(t)\|_{B^{s-1}_{p,r}}
\leq
C_T\Big( 2^{-k}+\int_0^t \|m^{(k,l)}(\tau)\|_{B^{s-1}_{p,r}}\,\d\tau \Big) .
\end{align*}
Note that, for $k=0$,
\begin{align*}
\|m^{(1,l)}(t)\|_{B^{s-1}_{p,r}} =\|m^{(l+1)}(t) -m^{(0)}(t)\|_{B^{s-1}_{p,r}}
=\|m^{(l+1)}(t)\|_{B^{s-1}_{p,r}} &\leq C_T .
\end{align*}
Then, by induction on $k$,
we obtain
\begin{align*}
\|m^{(k+1,l)}(t)\|_{B^{s-1}_{p,r}}
&\leq
C_T 2^{-k} \sum\limits_{i=0}^{k} \frac{(2C_T t)^{i}}{i!}
\\
&\leq
C_T\Big( 2^{-k} e ^{2C_T t} +2 \frac{(C_T t)^{(k+1)}}{(k+1)!}  \Big)
\\
&\leq
C_T\Big( 2^{-k} e ^{2C_T T} +2 \frac{(C_T T)^{(k+1)}}{(k+1)!}  \Big)
\end{align*}
which establishes
$\|m^{(k+1,l)}(t)\|_{B^{s-1}_{p,r}} \rightarrow 0$ uniformly for  $l\in\Nnum$,
as $k\rightarrow \infty$.

So far,
we have proved that
$\{m^{(k)}\}_{k\in \mathbb N}$ is bounded in $\mathcal C([0,T];B^s_{p,r})$ with the bound $C_T$,
and that
$\{m^{(k)}\}_{k\in \mathbb N}$ is a Cauchy sequence  in $\mathcal C([0,T];B^{s-1}_{p,r})$ and  converges to $m$ in the space  $\mathcal C([0,T];B^{s-1}_{p,r}) \subset L^\infty([0,T];\mathcal S')$.
Therefore, by Fatou's lemma (see \ref{lem:Besov.spaces}(vii)), we obtain
\begin{align}\label{Fatou-bound}
\|m\|_{L_T^\infty(B^{s}_{p,r})}\leq \frac{C^2\|m_0\|_{B^{s}_{p,r}}}{\sqrt{1-4C^3\|m_0\|_{B^s_{p,r}}^2 T}}
\end{align}
Finally, we can show that $m$ solves equation \eqref{ivp-transport}
by the following steps.

Return to the linear transport system \eqref{linear.transport.sys} and consider
\begin{align*}
\begin{aligned}
& \int_0^t\Big(\langle m^{(k+1)},\partial_t\phi(\tau)\rangle
+\langle m^{(k+1)}(\tau)J^{(k)}(\tau),\partial_x\phi(\tau)\rangle
+\langle m^{(k+1)}(\tau)J_x^{(k)}(\tau)-\langle K^{(k)}(\tau)m^{(k)}(\tau),\phi(\tau)\rangle
\Big)\d\tau
\\&
=\langle m^{(k+1)}(t),\phi(t)\rangle-\langle S_{k+1}m_0,\phi(0)\rangle
\end{aligned}
\end{align*}
for any test function $\phi\in \mathcal  C^1([0, T];\mathcal  S)$.
Applying Proposition~\ref{testfunction},
we can take the limit as $k\rightarrow\infty$
and examine each term on the left-hand side.
For instance,
we see
\begin{align*}
&\left|\int_0^t\langle m^{(k+1)}J_x^{(k)}, \phi\rangle\d t'-\int_0^t\langle mJ_x, \phi\rangle\d\tau\right|\\
&=\left|\int_0^t\langle m^{(k+1)}J_x^{(k)}-mJ_x, \phi\rangle\d\tau\right|\\
&\leq \int_0^t\left|\langle (m^{(k+1)}-m)J_x^{(k)}, \phi\rangle\right|\d\tau
+\int_0^t\left|\langle m(J^{(k)}-J)_x, \phi\rangle\right|\d\tau\\
&\leq T\Big(\|(m^{(k+1)}-m)J^{(k)}\|_{L^\infty_T(B^s_{p,r})}\|\phi\|_{L^\infty_T(B^{-s}_{p',r'})}
+\|m(J^{(k)}-J)_x\|_{L^\infty_T(B^s_{p,r})}\|\phi\|_{L^\infty_T(B^{-s}_{p',r'})}\Big)\\
&\leq C_T \big( \|m^{(k+1)}-m\|_{L^\infty_T(B^s_{p,r})} + \|m^{(k)}-m\|_{L^\infty_T(B^s_{p,r})} \big)\|\phi\|_{L^\infty_T(B^{-s}_{p',r'})}
\rightarrow 0 \text{ as } k\rightarrow \infty,
\end{align*}
which shows that
$\lim\limits_{k\rightarrow\infty} \int_0^t\langle m^{(k)}J_x^{(k)}, \phi\rangle\d\tau=\int_0^t\langle mJ_x, \phi\rangle\d\tau$.
The other terms can be examined by similar computations.
Therefore, for every $t\in [0,T]$, we have
\begin{align*}
\int_0^t\Big(\langle m,\partial_t\phi\rangle+\langle mJ,\partial_x\phi\rangle
+\langle mJ_x,\phi\rangle-\langle Km,\phi\rangle\Big)\d\tau
=\langle m(t),\phi(t)\rangle-\langle m_0,\phi(0)\rangle,
\end{align*}
and consequently $m$ solves equation \eqref{ivp-transport}.
Lemma~\eqref{regularity} then shows that $m$ belongs to
$\mathcal C([0,T];B^{s}_{p,r})$ (respectively $\mathcal C_w([0,T];B^{s-1}_{p,r})$) if $r<\infty$ (respectively $r=\infty$).

Noting that $m_t=Km-Jm_x\in\mathcal C([0,T];B^{s-1}_{p,r})$ (respectively $ L^\infty([0,T];B^{s-1}_{p,r})$) if $r<\infty$ (respectively $r=\infty$),
we conclude that $m\in E^s_{p,r}(T)$.

\subsection{Uniqueness and continuity with respect to the initial data}

Our first goal in this subsection will be to establish uniqueness.

Let $m$ and $\tilde m$ be two solutions to the peakon equation \eqref{m.transporteqn}
with initial data $m_{0}$ and $\tilde m_{0}$:
\begin{align*}
m_{t}+Jm_{x}=Km,
\quad
m|_{t=0}=m_{0},
\\
\tilde m_{t}+\tilde J\tilde m_{x}=\tilde K \tilde m,
\quad
\tilde m|_{t=0}=\tilde m_{0} .
\end{align*}
Taking the difference, we obtain
\begin{align}
(m-\tilde m)_t+J(m-\tilde m )_x=(\tilde J-J)\tilde m_x+Km-\tilde K\tilde m,
\quad
(m-\tilde m)|_{t=0}=m_{0}-\tilde m_{0}.
\end{align}
By computations similar to the ones in the previous subsection,
we can show that, for a.e. $t\in [0,T]$,
\begin{align*}
& e^{-C\gamma(t)}\|(m-\tilde m)(t)\|_{B^{s-1}_{p,r}}\\
& \leq \|m_{0}-\tilde m_{0}\|_{B^{s-1}_{p,r}}+
C\int^t_0 e^{-C\gamma(\tau)}  \|(\tilde J-J)(\tau)\tilde m_x(\tau)+K(\tau)m(\tau)-\tilde K(\tau)\tilde m(\tau)\|_{B^{s-1}_{p,r}}\d \tau
\end{align*}
where
\begin{align*}
\gamma(t) := \int_0^t\|m(\tau)\|^2_{B^{s}_{p,r}}\d\tau .
\end{align*}
Estimates similar to \eqref{Jterms.estimate} and \eqref{Kterms.estimate} yield
\begin{align*}
\|(\tilde J-J)\tilde m_x+Km-\tilde K\tilde m\|_{B^{s-1}_{p,r}}\leq
 \|m-\tilde m\|_{B^{s-1}_{p,r}}\big( \|m\|^2_{B^{s-1}_{p,r}}+\|\tilde m\|^2_{B^{s-1}_{p,r}} \big) .
\end{align*}
According to Gronwall's lemma, we see that, for a.e. $t\in [0,T]$,
\begin{align*}
\|(m-\tilde m)(t)\|_{B^{s-1}_{p,r}}
& \leq \|m_{0}-\tilde m_{0}\|_{B^{s-1}_{p,r}}
\exp\left(C\!\!\int^t_0\left(\|m(\tau)\|^2_{B^{s}_{p,r}}+\|\tilde m(\tau)\|^2_{B^{s}_{p,r}}\right)\d \tau\right)
\\
& \leq \|m_{0}-\tilde m_{0}\|_{B^{s-1}_{p,r}}
\exp\left(
\frac{C^5 T\|m_0\|^2_{B^{s}_{p,r}}}{1-4C^3\|m_0\|_{B^s_{p,r}}^2 T}
+\frac{C^5 T\|\tilde m_0\|^2_{B^{s}_{p,r}}}{1-4C^3\|\tilde m_0\|_{B^s_{p,r}}^2 T}
\right)
\end{align*}
through use of the bound \eqref{Fatou-bound}.
This establishes the uniqueness of solution $m$.

Our final goal will be to prove continuous dependence of $m$ on the initial data $m_0$.
Let $E^{s}_{p,r}(T^*)$ denote the space \eqref{wellposedness}.

For any $q\leq s-1$, since $B^{s-1}_{p,r}\hookrightarrow B^{q}_{p,r}$,
we know
\begin{align}\label{holder1}
\|(m-\tilde m)(t)\|_{B^{q}_{p,r}}\leq C\|(m-\tilde m)(t)\|_{B^{s-1}_{p,r}}.
\end{align}
For $s-1<q<s$, we use the interpolation Lemma~\eqref{interpolation1},
with $q=(s-q)(s-1)+(q+1-s) s$ where $(s-q)+(q+1-s)=1$ and $s-q,\,q+1-s\in (0,1)$,
which gives
\begin{equation}\label{holder2}
\begin{aligned}
\|(m-\tilde m)(t)\|_{B^{q}_{p,r}}&\leq\|(m-\tilde m)(t)\|^{s-q}_{B^{s-1}_{p,r}}\|(m-\tilde m)(t)\|^{q+1-s}_{B^{s}_{p,r}}
\\
&\leq (2C_T)^{q+1-s}\|(m_{0}-\tilde m_{0})(t)\|^{s-q}_{B^{s-1}_{p,r}}e^{2(s-q)C_T\, T}.
\end{aligned}
\end{equation}
These estimates  \eqref{holder1} and \eqref{holder2}
ensure the H\"older continuity of the solution map
from the initial data space $B^s_{p,r}$ to the space $E^{s'}_{p,r}(T)$ for any $s'< s$.
We also need to show that this map is also continuous from $B^s_{p,r}$ to $E^{s}_{p,r}$.
To proceed, we will need to introduce the following lemma
(proved in Ref~\cite{LiY}).

\begin{lemm}\label{allez38}
Let $1\leq p\leq \infty, 1\leq r<\infty, \sigma >1+\frac{1}{p}$ (or $\sigma =1+\frac{1}{p}, r=1, 1\leq p<\infty$).
Denote $\mathbb {\bar N}=\Nnum \cup\{\infty\}$.
Let $\{v^n\}_{n\in \mathbb{\bar N}}\subset \C([0,T];B^{\sigma-1}_{p,r})$.
Assume that $v^n$ is the solution to
\begin{align*}
\partial_t v^n+a^n\partial_xv^n=f,\\
v^n|_{t=0}=v_0,
\end{align*}
with $v_0\in B^{\sigma-1}_{p,r}, f\in L^1(0,T;B^{\sigma-1}_{p,r})$.
Also assume that, for some $g \in L^1(0,T)$,
\begin{align*}
\sup\limits_{n\in  \mathbb {\bar N}}\|a^n\|_{B^\sigma_{p,r}}\leq g(t) .
\end{align*}
If $a^n\rightarrow a^\infty$ in  $L^1(0,T;B^{\sigma-1}_{p,r})$ when $n\rightarrow \infty$,
then
$v^n\rightarrow v^\infty$ in $\C([0,T];B^{\sigma-1}_{p,r})$  when $n\rightarrow \infty$.
\end{lemm}

Now, for all $n\in\mathbb{\bar N}$,
suppose $m_n\in \C([0,T];B^s_{p,r})$ is the solution to equation \eqref{unifiedfamilyeqn}
with initial data $m_{n,0}\in B^s_{p,r}$:
\begin{align*}
\partial_tm_n+J_{n}\partial_xm_n= K_{n}m_n,\\
m_n(t,x)|_{t=0}=m_{n,0}(x) .
\end{align*}

\begin{prop}\label{continuity-initial-data}
For $1\leq r<\infty$ (or $r=\infty$),
if $m_{n,0}\rightarrow m_{\infty,0}$ in $B^s_{p,r}$ as $ n\rightarrow \infty$,
then $m_n\rightarrow m_\infty$ in $\C([0,T];B^s_{p,r})$  (or $C_w([0,T];B^{s}_{p,r})$).
Here $T$ is a positive number satisfying
\begin{equation*}
4C^3\sup\limits_{n\in\mathbb{\bar N}}\|m_{n,0}\|^2_{B^s_{p,r}}T< 1 .
\end{equation*}
\end{prop}

\begin{proof}
Let $f_n = K_n m_n$, and decompose $m_n=y_n+z_n$ such that
\begin{align*}
\partial_t y_n+J_{n}\partial_xy_n=f_\infty,
\\
y_n|_{t=0}=m_{\infty,0},
\end{align*}
and
\begin{align*}
\partial_t z_n+J_n\partial_xz_n=f_n-f_\infty,
\\
z_n|_{t=0}=m_{n,0}-m_{\infty,0}.
\end{align*}

When $1\leq r<\infty$,
the bound \eqref{Fatou-bound} and the uniqueness of the solution to equation \eqref{unifiedfamilyeqn}
imply that
\begin{align*}
\|m_n\|_{L^\infty_T(B^s_{p,r})}\leq M
\end{align*}
where
\begin{align*}
M:=
\frac{C^2\sup_{n\in\mathbb{\bar N}}\|m_{n,0}\|_{B^{s}_{p,r}}}
{\sqrt{ 1-4C^3\sup_{n\in\mathbb{\bar N}}\|m_{n,0}\|_{B^{s}_{p,r}}T }}.
\end{align*}
Hence, we see $m_n\in C([0,T];B^s_{p,r})$.
This implies that $u_n\in C([0,T];B^{s+2}_{p,r})$ and $\partial_xu_n\in C([0,T];B^{s+1}_{p,r})$,
and consequently, we have
$K_n= i \Im(e^{i\theta}Q_n)  -\Re(e^{i\theta}Q_{n,x})\in C([0,T];B^s_{p,r})$.
Thus $f_n = K_n m_n$ is uniformly bounded in $\C([0,T];B^s_{p,r})$.

Likewise,
\begin{align*}
\|J_{n}\|_{B^{s+1}_{p,r}}=\|\Re\left(e^{i\theta}(u_n+u_{n,x})(\bar u_n-\bar u_{n,x})\right)\|_{B^{s+1}_{p,r}}
\leq C\|m_n\|^2_{B^s_{p,r}}\leq CM^2.
\end{align*}
Hence, we have
\begin{align*}
\|(J_n-J_\infty)(t)\|_{B^{s+1}_{p,r}}
\leq C\|(m_n-m_\infty)(t)\|_{B^{s-1}_{p,r}}.
\end{align*}
Since $s-1<s$,
the estimate \eqref{holder1} gives
$m_n\rightarrow m_\infty$ in $L^1(0,T;B^{s-1}_{p,r})$,
and therefore $J_n\rightarrow J_\infty$ in $L^1(0,T;B^{s+1}_{p,r})$.
Then by Lemma~\ref{allez38} with $\sigma=s+1$, we obtain
\begin{equation}\label{allez16}
y_n\rightarrow m_\infty \text{ in $\C([0,T];B^s_{p,r})$ as $n\rightarrow \infty$}.
\end{equation}

To control $z_n$,  we need to estimate  $ f_n-f_\infty$:
\begin{align*}
\|f_n-f_\infty\|_{B^{s}_{p,r}}&=\|K_{n}m_n-K_{\infty}m_\infty\|_{B^s_{p,r}}\\
&\leq C(\|m_n\|^2_{B^s_{p,r}}+\|m_\infty\|^2_{B^s_{p,r}})\|m_n-m_\infty\|_{B^s_{p,r}}.
\end{align*}
Applying Lemmas  \ref{apriori-estimates}, \ref{Luowei} in Appendix,
and using computations similar to previous ones, we have
\begin{align*}
\|z_n\|_{B^s_{p,r}}&\leq e^{C\!\!\int^t_0\|m_n\|^2_{B^s_{p,r}}\d t'}\Big(\|m_{n,0}-m_{\infty,0}\|_{B^s_{p,r}}+C\int^t_0(\|m_n\|^2_{B^s_{p,r}}
+\|m_\infty\|^2_{B^s_{p,r}})(\|m_n-m_\infty\|_{B^s_{p,r}})\d t' \Big)\\
&\leq C M^2e^{CM^2T}\Big(\|m_{n,0}-m_{\infty,0}\|_{B^s_{p,r}}
+\int^t_0\|m_n-m_\infty\|_{B^s_{p,r}}\d t' \Big).
\end{align*}
From \eqref{allez16}, for any $\varepsilon >0$,
we can choose a sufficiently large $n$ such that
\begin{align*}
\|y_n-m_\infty\|_{B^s_{p,r}} < \varepsilon .
\end{align*}
Thus,
\begin{align*}
\|m_n-m_\infty\|_{B^s_{p,r}}
&
\leq \|z_n\| +\|y_n-m_\infty\|
\\&
\leq
\varepsilon + C M^2e^{CM^2T}
\Big(\|m_{n,0}-m_{\infty,0}\|_{B^s_{p,r}}
+\int^t_0\|m_n(\tau) -m_\infty(\tau)\|_{B^s_{p,r}}\d\tau \Big).
\end{align*}
We now use Gronwall's inequality to get
\begin{align*}
\|(m_n-m_\infty)(t)\|_{B^s_{p,r}}\leq \tilde C\Big(\varepsilon + \|m_{n,0}-m_{\infty,0}\| _{B^s_{p,r}}\Big)\quad \text{for a.e.}~  t\in[0,T],
\end{align*}
for some constant $\tilde C =\tilde C(s,p,R ,M,T)$.
This establishes the continuity of  \eqref{unifiedfamilyeqn} in $C([0,T];B^{s}_{p,r})$
with respect to the initial data in $B^s_{p,r}$ for $r<\infty$.

When $r=\infty$, by the inequality \eqref{holder1}, we see that  $\|m_n-m_\infty\|_{L_T^\infty(B^{s-1}_{p,r})}$
tends to $0$ as $n\rightarrow \infty$. Hence  for fixed $\phi \in B^{-s}_{p',1}$, we have
\begin{align*}
\langle m_n-m_\infty, \phi\rangle&=
\langle S_j[m_n-m_\infty], \phi\rangle-\langle (\mathrm{Id}-S_j)[m_n-m_\infty], \phi\rangle\\
&=\langle m_n-m_\infty, S_j\phi\rangle+\langle m_n-m_\infty, (\mathrm{Id}-S_j)\phi\rangle.
\end{align*}
Applying Proposition~\ref{testfunction} in the Appendix, we have
\begin{equation}\label{allez13}
|  \langle m_n-m_\infty, (\mathrm{Id}-S_j)\phi\rangle|
 \leq  2CM\|\phi-S_j\phi\|_{B^{-s}_{p',1}},
\end{equation}
and
\begin{equation}\label{allez14}
| \langle m_n-m^\infty, S_j\phi\rangle|
\leq  C\|m_n-m_\infty\|_{L_T^\infty (B^{s-1}_{p,r})}\|S_j\phi\|_{B^{1-s}_{p',1}}.
\end{equation}
Note that $\|\phi-S_j\phi\|_{B^{-s}_{p',1}}$ tends to zero as $j\rightarrow \infty$
and that $\|m_n-m_\infty\|_{L_T^\infty (B^{s-1}_{p,r})}$ tends to zero as $n\rightarrow \infty$.
Then the right hand-side of  \eqref{allez13} will be arbitrarily small  for $j$ large enough.
For such fixed $j$, we let $n$ go to infinity so that the right hand-side of \eqref{allez14} tends zero.
Thus, we conclude that $ \langle m_n(t)-m_\infty(t), \phi\rangle$ tends to zero as $n\rightarrow \infty$
for the case $r=\infty$.
\end{proof}

\section{$L^1$ conservation law}\label{sec:L1conservation}

The $L^1$ conservation law \eqref{L1.conservation} in Theorem~\ref{thm:L1conservation}
is a generalization of the same conservation law known \cite{CheLiuQuZha} for the FORQ/mCH equation,
and it can be proved by a similar method.

The main idea is to make a change of variables in the $L^1$ norm by means of
an increasing diffeomorphism of $\Rnum$ that arises from viewing $J$
in the transport form \eqref{m.transporteqn} of the peakon equation \eqref{unifiedfamilyeqn}
as a velocity field.
Thus, we begin by considering the initial value problem
\begin{equation}\label{ODE}
\begin{aligned}
& h_t(t,x)=J(t,h(t,x)),
\quad
0<t<T,
\\
& h(0,x)=x,
\quad
x\in\Rnum,
\end{aligned}
\end{equation}
where $J$ is expression \eqref{J} for the solution $u$ of the peakon equation
with initial data $m_0\in H^s$, with $s> \frac{1}{2}$,
and $T$ is the existence time of $u$.

From classical results in ODE theory,
the following properties of $h(t,x)$ can be readily proven.

\begin{lemm}\label{lem:incr.diffeo}
The ODE problem \eqref{ODE} has a unique solution $h\in\C^1([0,T)\times\Rnum)$,
with
\begin{align}\label{characteristic}
h_x(t,x)=\exp\left(\int_0^t J_x(\tau,h(\tau,x))\,\d\tau\right) >0 .
\end{align}
The resulting map $h(t,\cdot)$ is an increasing diffeomorphism of $\Rnum$
for all $t\in[0,T)$.
\end{lemm}

\begin{proof}
For $m_0\in H^s$,
we have $m\in\C([0,T);H^s)\cap \C^1([0,T);H^{s-1})$
since $s> \frac{1}{2}$,
and hence $u\in\C([0,T);H^{s+2})\cap \C^1([0,T);H^{s+1})$.
Therefore, $J$ is in $\C([0,T);H^{s+1})\cap \C^1([0,T);H^{s})$,
which is contained in the space of Lipschitz functions $\C([0,T);\C^{0,1})$.
According to ODE theory,
this implies there is a unique solution $h\in\C^1([0,T)\times\Rnum)$.

Differentiating \eqref{ODE} with respect to $x$ yields
\begin{equation}
\begin{aligned}
& h_{xt}(t,x)= J_x(t,h(t,x))h_x(t,x),
\quad
0<t<T,
\\
& h_x(0,x)=1,
\quad
x\in\Rnum.
\end{aligned}
\end{equation}
Solving this ODE problem with $h_x$ as the unknown function,
we obtain \eqref{characteristic}.
\end{proof}

We next show that the solution $m$ to the peakon equation \eqref{unifiedfamilyeqn}
along the line defined by $h(\cdot,x)$ is an oscillatory function of $t$.

\begin{lemm}\label{lem:m.along.h}
Let $m_0\in H^s$, with $s>\frac{1}{2}$,
and let $h(t,x)$ be the solution to the ODE problem \eqref{ODE}.
Let $T>0$ be the existence time of the solution $m(t,x)$ to the peakon equation \eqref{unifiedfamilyeqn}.
Then for $0\leq t<T$,
\begin{equation}\label{m.along.h}
m(t,h(t,x))=m_0(x)\exp\left(\int_0^t K(\tau,h(\tau,x))\,\d\tau\right)
\end{equation}
where $K$ is expression \eqref{K}.
\end{lemm}

\begin{proof}
Differentiating $m(t,h(t,x))$ with respect to $t$,
and using equations \eqref{unifiedfamilyeqn} and \eqref{ODE},
we get
\begin{equation}\label{mt.on.h}
\begin{aligned}
\frac{\d}{\d t}\big(m(t,h(t,x))\big)
&=m_t(t,h(t,x)) +m_x(t,h(t,x)) J(t,h(t,x)) \\
&=K(t,h(t,x)) m(t,h(t,x)) .
\end{aligned}
\end{equation}
Next we multiply \eqref{mt.on.h} by the integrating factor
$e^{-\int_0^t K(\tau,h(\tau,x))\,\d\tau}$
and integrate starting at $t=0$ with the initial condition
$m(0,h(0,x))= m(0,x)=m_0(x)$.
This yields \eqref{m.along.h}.
\end{proof}

\begin{rema}
Relation \eqref{m.along.h} shows that the solution has temporal oscillatory properties
since $K$ has a non-vanishing imaginary part.
It also can be used to establish finite propagation properties
for the peakon equation \eqref{unifiedfamilyeqn}.
\end{rema}

\subsection{Proof of Theorem~\ref{thm:L1conservation}}

Now we can prove conservation of the $L^1$ norm of $m$
via a change of variables using the increasing diffeomorphism $h(t,\cdot)$ for all $0\leq t<T$.

First, we have
\begin{align*}
m(t,h(t,x))h_x(t,x)
& = m_0(x)\exp\left(\int_0^t K(\tau,h(\tau,x))\,\d\tau\right)\exp\left(\int_0^t J_x(\tau,h(\tau,x))\,\d\tau\right)\\
& = m_0(x)\exp\left(i\int_0^t \Im(e^{i\theta}Q(\tau,h(\tau,x)))\,\d\tau\right)
\end{align*}
from relation \eqref{m.along.h} and equation \eqref{characteristic},
combined with expression \eqref{K}.
Thus,
$|m(t,h(t,x))h_x(t,x)|= |m_0(x)|$
by $U(1)$ invariance.
This allows the $L^1$ norm to be evaluated as
\begin{align*}
\|m\|_{L^1} = \int_{\Rnum}|m(t,x)|\d x
& =\int_{\Rnum}|m(t,h(t,x))h_x(t,x)|\d x\\
& =\int_{\Rnum}|m_0(x)|\d x
=\|m_0\|_{L^1} .
\end{align*}
Hence, Theorem~\ref{thm:L1conservation} is proved.

\section{Blow-up criterion}\label{sec:blowupcriteria}

We start by introducing the following blow-up criterion of integral form.

\begin{lemm}\label{lem:integral.blowup.criterion}
Let $m_0\in B^s_{p,r}$ with $ 1\leq p,r \leq\infty$ and $s>\max(\frac{1}{2},\frac{1}{p})$.
If the maximal existence time $T>0$ of
the corresponding solution $m$ to the peakon equation \eqref{unifiedfamilyeqn}
is finite, then
\begin{align*}
\int^T_0\|m(t)\|^2_{L^\infty}\d t=\infty .
\end{align*}
\end{lemm}

\begin{proof}
Argue by contradiction.
Suppose $\int_0^T\|m\|_{L^\infty}^2\d t<\infty$.
Noting $s>\max(1/2,1/p)>0$ and applying Lemma \ref{LiYin2},
we can obtain, for $\forall ~t\in[0,T),$
\begin{align*}
\|m(t)\|_{B^s_{p,r}}\leq \|m_0\|_{B^s_{p,r}}
+\int_0^t \|Km(\tau)\|_{B^s_{p,r}}\d \tau+C\int_0^t \big(\|m(\tau)\|_{B^s_{p,r}}\|J_x\|_{L^\infty}+
\|m(\tau)\|_{L^\infty}\|J_x\|_{B^s_{p,r}}\big)\d\tau.
\end{align*}
According to the product laws in Lemma (A.6) (1), we can see that
\begin{gather*}
\|Km\|_{B^s_{p,r}} \leq C(\|K\|_{L^\infty}\|m\|_{B^s_{p,r}}+\|K\|_{B^s_{p,r}}\|m\|_{L^\infty}),
\\
\|K\|_{L^\infty} \leq C\|m\|_{L^\infty}^2,
\quad
\|K\|_{B^s_{p,r}} \leq C\|m\|_{L^\infty}\|m\|_{B^s_{p,r}},
\\
\|J_x\|_{L^\infty} \leq C\|m\|_{L^\infty}^2,
\quad
\|J_x\|_{B^s_{p,r}} \leq C\|m\|_{L^\infty}\|m\|_{B^s_{p,r}}.
\end{gather*}
Therefore, we have
\begin{align*}
\|m(t)\|_{B^s_{p,r}}\leq \|m_0\|_{B^s_{p,r}}
+C\int_0^t \|m(\tau)\|_{B^s_{p,r}}\|m\|^2_{L^\infty}\d \tau.
\end{align*}
Gronwall's inequality then leads to
\begin{align*}
\|m(t)\|_{B^s_{p,r}}\leq \|m_0\|_{B^s_{p,r}}\exp\left( C\int_0^t \|m\|^2_{L^\infty}\d \tau
\right)<\infty
\end{align*}
for all $t\in[0,T)$.
Thus we can extend the solution $m$ beyond the maximal $T$, which is a contradiction.
\end{proof}

Combining this Lemma and Theorem~\ref{thm:wellposedness},
we readily obtain the following pointwise blow-up criterion.

\begin{coro}\label{cor:pointwise.blowup.criterion}
Let $ m_0 \in B^s_{p,r}$ with $ 1\leq p,r \leq\infty$ and $s>\max(\frac{1}{2},\frac{1}{p})$.
Let $T > 0 $ be the maximal existence time of the corresponding solution $m$ to the peakon equation \eqref{unifiedfamilyeqn}.
Then the solution blows up in finite time if and only if
\begin{equation}\label{pointwise.blowup.criterion}
\limsup\limits_{ t\uparrow T} \|m(t)\|_{L^\infty} = \infty .
\end{equation}
\end{coro}

\subsection{Proof of Theorem~\ref{thm:criterion}}

Now we can prove the precise blow-up criterion stated in Theorem~\ref{thm:criterion}.

Suppose $T$ is finite,
and assume for contradiction that the blow-up criterion \eqref{blowup.criterion}
does not hold.
Then there exists a constant $C\in \Rnum$ such that
\begin{equation}\label{Jx.condition}
J_x(t,x)\geq -C,
\quad
\forall (t,x)\in [0,T)\times \Rnum,
\end{equation}
where $J$ is given expression \eqref{J}.
This inequality yields the bound
\begin{equation}\label{Jx.bound}
\exp\left(-\int^t_0 J_x(\tau,h(\tau,x))\d\tau \right)
\leq e^{CT}.
\end{equation}

Using relation \eqref{m.along.h} in Lemma~\ref{lem:m.along.h},
and substituting expression \eqref{K},
we obtain
\begin{align*}
|m(t,h(t,x))| &=
\bigg|m_0(x)\exp\left( i\int_0^t\Im (e^{i\theta}Q(\tau,h(\tau,x))) \d\tau\right)
\exp\left(-\int_0^t J_x(\tau,h(\tau,x)) \d\tau\right)\bigg| \\
&\leq \|m_0\|_{L^\infty}\,e^{CT}
\end{align*}
due to \eqref{Jx.bound} and $U(1)$ invariance.
The embedding property $H^s\hookrightarrow L^{\infty}$ when $s>\frac{1}{2}$
then yields
\begin{align*}
\|m(t)\|_{L^\infty}\leq C_1 \|m_0\|_{B^s_{p,r}} e^{CT}
\end{align*}
for some positive constant $C_1$.
This shows $\limsup\limits_{ t\uparrow T} \|m(t)\|_{L^\infty}$ is finite,
and therefore $m$ does not blow up in finite time
according to Corollary~\ref{cor:pointwise.blowup.criterion}.

Hence, the inequality \eqref{Jx.condition} cannot hold if $m$ blows up in finite time,
which establishes Theorem~\ref{thm:criterion}.

\section{Blow-up phenomenon}\label{sec:blowup}

The blow-up phenomenon stated in Theorem~\ref{thm:blowup} is obtained
from adapting the transport method used for the FORQ/mCH equation in \Ref{CheLiuQuZha}.
This method relies on
the blow-up criterion \eqref{blowup.criterion} from Theorem~\ref{thm:criterion},
the $L^1$ conservation law \eqref{L1.conservation},
and the following basic estimates.

\begin{lemm}\label{lem:u.ux.estimate}
Let $m_0\in H^s\cap L^1(\Rnum)$ with $s>\frac{1}{2}$.
Suppose that $T$ is the maximal existence time of the corresponding solution $m$
to the peakon equation \eqref{unifiedfamilyeqn}.
Then for all $t\in [0, T)$:
\begin{equation}\label{u.ux.estimate}
\begin{aligned}
|u(t,x)|\leq  \tfrac{1}{2} \|m_0\|_{L^1},
\quad
|u_x(t,x)|\leq  \tfrac{1}{2} \|m_0\|_{L^1} .
\end{aligned}
\end{equation}
\end{lemm}

\begin{proof}
The estimates
$|u(t,x)|\leq \tfrac{1}{2} \|m\|_{L^1}$ and $|u_x(t,x)|\leq \tfrac{1}{2} \|m\|_{L^1}$
are well known in the case of real functions by Young's inequality for convolution.
They extend directly to complex functions.
Then the $L^1$ conservation law \eqref{L1.conservation} completes the proof.
\end{proof}

To proceed, we will first need the transport equations of $u\pm u_x$
under the flow produced by viewing $J$ as a velocity field.

\begin{prop}\label{prop:u-expression1}
Let $u$ be a strong solution to the peakon equation \eqref{unifiedfamilyeqn}.
Write $\Delta:=1-\partial_x^2$ and $v^\pm:=u\pm u_x$.
Then $v^\pm$ satisfies the transport equation
\begin{equation}\label{v.J.flow}
v^\pm_t+Jv^\pm_x =\Delta^{-1}(1\pm \partial_x)\big( -J_xu+i\Im(K) m \big) \mp\Delta^{-1}\big( J_xu_x \big) :=L^\pm .
\end{equation}
\end{prop}

\begin{proof}
We start from the identity
\begin{equation*}
\Delta u_t+\Delta(Ju_x)
=m_t+Jm_x-(J_{x}u_x)_x-J_xu_{xx} .
\end{equation*}
Using the transport form \eqref{m.transporteqn} of the peakon equation,
we get
\begin{align*}
\Delta u_t+\Delta(Ju_x)
&=-J_xm+i\,\Im(K)m-(J_{x}u_x)_x-J_xu_{xx}\\
&=-J_xu-(J_xu_x)_x+i\,\Im(K)m,
\end{align*}
which leads immediately to the transport equation
\begin{equation}\label{u.J.flow}
u_t+Ju_x =\Delta^{-1}\big( -J_xu-(J_xu_x)_x +i \Im(K)m \big) .
\end{equation}
Then we differentiate with respect to $x$ and rearrange the terms to get
\begin{equation}\label{ux.J.flow}
u_{xt}+Ju_{xx}=\Delta^{-1}\partial_x\big( -J_xu+i \Im(K)m \big) -\Delta^{-1}\big( J_xu_x \big) .
\end{equation}
Adding and subtracting \eqref{u.J.flow} and \eqref{ux.J.flow}
yields equation \eqref{v.J.flow}.
\end{proof}

The main transport equation underlying the blow-up phenomenon
will be the resulting flow on $J_x$.
We note, from expressions \eqref{Q}, \eqref{J} and \eqref{Qx}, that
\begin{equation}\label{J.rel}
J= \Re(e^{i\theta}Q) = \Re(e^{i\theta}v^+ \bar v^-)
\end{equation}
and
\begin{equation}\label{Jx.rel}
J_x= \Re(e^{i\theta}Q_x) = \Re(e^{i\theta}( v^+\bar m - \bar v^- m )) .
\end{equation}

\begin{prop}\label{prop:Jx.transport}
Let $u$ be a strong solution to the peakon equation \eqref{unifiedfamilyeqn}.
Then $J_x$ satisfies the transport equation
\begin{equation}\label{Jx.transport}
J_{xt}+JJ_{xx}+J^2_x =
\Im(e^{i\theta}Q) \Im(e^{i\theta}( v^+\bar m +\bar v^- m ))
+ \Re( e^{i\theta}(\bar m L^+ -m\bar L^-)) .
\end{equation}
\end{prop}

\begin{proof}
Differentiating \eqref{Jx.rel} and using the identity $v^\pm_x = \pm(v^\pm -m)$,
we have
\begin{align*}
J_{xx} & =\Re(e^{i\theta}( v^+_x\bar m - \bar v^-_x m + v^+\bar m_x - \bar v^- m_x ))
\end{align*}
and
\begin{align*}
J_{xt}= & \Re(e^{i\theta}( v^+_t\bar m - \bar v^-_t m + v^+\bar m_t - \bar v^- m_t ))\\
= & \Re(e^{i\theta}( (-Jv^+_x + L^+)\bar m - (-J\bar v^-_x + \bar L^-) m
+v^+(\bar K\bar m - J\bar m_x) - \bar v^- (K m - J m_x) ))\\
=& -J\Re(e^{i\theta}( v^+_x\bar m - \bar v^-_x m +v^+\bar m_x  - \bar v^- m_x) )
+ \Re(K) \Re(e^{i\theta}( v^+\bar m - \bar v^- m ))
\\&
+\Im(K) \Im(e^{i\theta}( v^+\bar m +\bar v^- m ))
+\Re(e^{i\theta}( L^+\bar m - \bar L^- m ))
+\Re(e^{i\theta}( L^+\bar m - \bar L^- m ))
\end{align*}
after use of the transport equations \eqref{v.J.flow} and \eqref{m.transporteqn}.
Substituting expression \eqref{K} then yields \eqref{Jx.transport}.
\end{proof}

Now we will obtain a pointwise estimate for the right-hand side of the transport equation \eqref{Jx.transport}.

\begin{lemm}\label{lem:Jxeqn.estimate}
Let $m_0\in H^s\cap L^1$, with $s>\frac{1}{2}$,
and let $T$ be the maximal existence time of the corresponding solution
$m$ to the peakon equation \eqref{unifiedfamilyeqn}.
Then there exists a constant $C_0=7\|m_0\|_{L^1}^3$ such that, for all $t\in[0,T)$:
\begin{equation}\label{Jxeqn.estimate}
J_{xt}+JJ_{xx}+J^2_{x} \leq C_0|m|.
\end{equation}
\end{lemm}

\begin{proof}
From the pointwise estimates \eqref{u.ux.estimate},
we have
\begin{equation*}
|v^\pm| \leq \|v^\pm\|_{L^\infty} \leq \|m_0\|_{L^1} .
\end{equation*}
Hence, we can easily estimate
\begin{gather*}
|\Im(e^{i\theta}Q)|
\leq |v^+| |\bar v^-| \leq \|m_0\|_{L^1}^2 ,
\\
|\Im(e^{i\theta}(v^+\bar m +\bar v^- m ))|
\leq (|v^+|+|v^-|) |m|
\leq 2\|m_0\|_{L^1} |m| ,
\end{gather*}
and
\begin{equation}
\|J_x\|_{L^1} \leq \|v^+\bar m -\bar v^- m\|_{L^1}
\leq (\|v^+\|_{L^\infty}+\|v^-\|_{L^\infty})\|m\|_{L^1}
\leq 2\|m_0\|_{L^1}^2 .
\end{equation}
Then, on the right-hand side of equation \eqref{Jx.transport},
the first term can be estimated as
\begin{equation}\label{Jxeqn.term1.estimate}
| \Im(e^{i\theta}Q) \Im(e^{i\theta}( v^+\bar m +\bar v^- m )) |
\leq 2\|m_0\|_{L^1}^3 |m| ,
\end{equation}
while for the second term we have
\begin{equation}\label{Jxeqn.term2.estimate}
| \Re( e^{i\theta}(\bar m L^+ -m\bar L^-)) |
\leq (|L^+| +|L^-|) |m|
\end{equation}
and
\begin{equation*}
|L^\pm| \leq
\|\Delta^{-1}(J_xu_x)\|_{L^\infty}
+\|\Delta^{-1}(1\pm \partial_x)(J_xu)\|_{L^\infty}
+\|\Delta^{-1}(1\pm \partial_x)(\Im(e^{i\theta}Q) m)\|_{L^\infty} .
\end{equation*}
We now estimate each of these three terms
by using $\Delta^{-1} = \frac{1}{2} e^{-|x|} *\,$.
For the first term, we have
\begin{align*}
\|\Delta^{-1}(J_xu_x)\|_{L^\infty}
=\tfrac{1}{2}\|e^{-|x|}*(J_xu_x)\|_{L^\infty}
\leq \tfrac{1}{2} \|J_xu_x\|_{L^1}
\leq \tfrac{1}{2} \|u_x\|_{L^\infty} \|J_x\|_{L^1}
\leq \tfrac{1}{2} \|m_0\|^3_{L^1} .
\end{align*}
Similarly, the second and third terms can be estimated as
\begin{align*}
\|\Delta^{-1}(1\pm \partial_x)(J_xu)\|_{L^\infty}
& =\tfrac{1}{2}\|(1\pm \partial_x)(e^{-|x|}*(J_xu))\|_{L^\infty}\\
&\leq \tfrac{1}{2}\|(1\mp \sgn(x)) e^{-|x|}\|_{L^\infty} \|J_xu\|_{L^1}
\leq \|u\|_{L^\infty} \|J_x\|_{L^1}
\leq \|m_0\|^3_{L^1}
\end{align*}
and
\begin{align*}
\|\Delta^{-1}(1\pm \partial_x)(\Im(e^{i\theta}Q) m)\|_{L^\infty}
& =\tfrac{1}{2}\|(1\pm \partial_x)(e^{-|x|}*(\Im(e^{i\theta}v^+ \bar v^-) m))\|_{L^\infty} \\
& \leq \tfrac{1}{2}\|(1\mp \sgn(x)) e^{-|x|}\|_{L^\infty} \|\Im(e^{i\theta}v^+ \bar v^-) m\|_{L^1}\\
& \leq \|v^+\|_{L^\infty} \|\bar v^-\|_{L^\infty} \|m\|_{L^1}
\leq \|m_0\|^3_{L^1} .
\end{align*}
Hence, we obtain
\begin{equation}
|L^\pm| \leq \tfrac{5}{2} \|m_0\|^3_{L^1} .
\end{equation}
Combining this estimate with \eqref{Jxeqn.term2.estimate} and \eqref{Jxeqn.term1.estimate}
then yields \eqref{Jxeqn.estimate}.
\end{proof}

From Lemma~\ref{lem:Jxeqn.estimate},
we can now derive the sufficient condition for blow-up stated in Theorem~\ref{thm:blowup}.
The derivation uses the type of argument employed in \Ref{GuiLiuOlvQu,CheLiuQuZha}
for blow-up of the FORQ/mCH equation.

\subsection{Proof of Theorem~\ref{thm:blowup}}

First we evaluate the pointwise estimate \eqref{Jxeqn.estimate}
on the increasing diffeomorphism $h(t,\cdot)$ given by Lemma~\ref{lem:incr.diffeo},
which yields
\begin{equation}\label{blowup.eqn}
J_{xt}(t,h(t,x))+(JJ_{xx})(t,h(t,x))+J_x(t,h(t,x))^2\leq C_0|m(t,h(t,x))| .
\end{equation}
Since $J(t,h(t,x))=h_t(t,x)$ from equation \eqref{ODE},
we have
\begin{align*}
J_{xt}(t,h(t,x))+(JJ_{xx})(t,h(t,x))+J_x(t,h(t,x))^2
= \frac{\d}{\d t}\big(J_x(t,h(t,x))\big)+J_x(t,h(t,x))^2
\end{align*}
for the left-hand side of \eqref{blowup.eqn}.
On the right-hand side of \eqref{blowup.eqn},
we use the relation \eqref{m.along.h} in Lemma~\ref{lem:m.along.h} to get
\begin{align*}
|m(t,h(t,x))| = |m_0(x)|\exp\left(-\int_0^t J_x(\tau,h(\tau,x))\d \tau\right) .
\end{align*}
Hence we obtain
\begin{align*}
\frac{\d}{\d t}\big(J_x(t,h(t,x))\big)+J_x(t,h(t,x))^2
\leq C_0|m_0(x_0)|\exp\left(-\int_0^t J_x(\tau,h(\tau,x))\d \tau\right)
\end{align*}
which can be rearranged into the form
\begin{align*}
\frac{\d^2}{\d t^2} \exp\left(\int_0^t J_x(\tau,h(\tau,x))\d \tau\right)
\leq C_0|m_0(x)| .
\end{align*}
By integrating this inequality, we get
\begin{equation}\label{blowup.rel}
\exp\left(\int_0^t J_x(\tau,h(\tau,x))\d \tau\right)
\leq 1+J_x(0,x)\, t+ \tfrac{1}{2} C_0|m_0(x)|\,t^2 .
\end{equation}
This inequality can be used to establish blow-up.

Choose a point $x_0\in\Rnum$ such that the quadratic polynomial in $t$ given by
the right-hand side of \eqref{blowup.rel}
has a positive discriminant:
\begin{align*}
J_x(0,x_0)^2 -2C_0|m_0(x_0)| >0,
\quad
m_0(x_0)\neq 0 .
\end{align*}
Then the polynomial has two real roots
\begin{align*}
0<T_1=\frac{-J_x(0,x_0)-\sqrt{J_x(0,x_0)^2 -2C_0|m_0(x_0)|}}{C_0|m_0(x_0)|}
\leq T_2= \frac{-J_x(0,x_0)+\sqrt{J_x(0,x_0)^2 -C_0|m_0(x_0)|}}{C_0|m_0(x_0)|} ,
\end{align*}
whereby
\begin{align*}
1+J_x(0,x_0)\,t+\frac {C_0|m_0(x_0)|}2\,t^2\searrow 0,  \text{ as } t\nearrow T_1 .
\end{align*}
Consequently, from the inequality \eqref{blowup.rel}, we have
\begin{align*}
0<\exp\left( \int_0^t \inf\limits_{x\in\Rnum}J_x(\tau,x)\,\d\tau\right)
\leq \exp\left( \int_0^t J_x(\tau,h(\tau,x_0))\, \d\tau\right)
\rightarrow 0 \text{ as } t\rightarrow T_1
\end{align*}
which implies
\begin{align*}
\limsup\limits_{ t\uparrow T} \inf\limits_{x\in\Rnum}J_x(t,x)=-\infty .
\end{align*}

According to Theorem~\ref{thm:criterion}, $m$ blows up at the finite time $T_1$.
This completes the proof.

\section{Concluding Remarks and open problems}\label{sec:conclude}

In this paper, we have studied the Cauchy problem for integrable $U(1)$-invariant peakon equations generated from the NLS hierarchy.
Main results include local well-posedness in Theorem~\ref{thm:wellposedness}, $L^1$ conservation law in Theorem~\ref{thm:L1conservation}, 
and blow-up scenarios in Theorem~\ref{thm:criterion} and Theorem~\ref{thm:blowup}.
The NLS hierarchy is a very important integrable family in soliton theory, and many related work such as the Riemann-Hilbert (RH) problem and $\bar{\partial}$-approach has been done in recent years. But, for the $U(1)$-invariant peakon system (1.7), peakon stability and long time asymptotic behavior associated with the RH problem and $\bar{\partial}$-approach are still open, which is looked forward to solving elsewhere.

\section*{Acknowledgements}

S.C.A.~is supported by an NSERC Discovery grant.
H.H.~ Thanks 
 the Department of Mathematics \& Statistics at Brock University 
for its support during the period in which this work was started.
Z.Q. thanks the UT President Endowed Professorship (Project \# 450000123)  and the 2019--2020 Hunan overseas distinguished professorship project (No.~2019014) for their partial support.

\appendix
\section{Besov spaces, Littlewood-Paley decomposition,\\ and linear transport theory}\label{sec:technicalbackground}

We will recall some facts on the Littlewood-Paley decomposition, the Besov spaces and some of their useful properties,
and we will also summary the main results that will be needed from the linear transport theory.
For more details, see \Ref{BahCheDan}.

Throughout, we use the notation:\\
$\mathcal{S}$ is the Schwartz space of fast decrease;
$\mathcal{S}'$ is the tempered distribution space. \\
$\mathcal{F}$ is the Fourier transform $\hat f(\xi)=\mathcal{F}f:=\int_{\mathbb R}e^{-ix\xi}f(x)\d x$;
$\mathcal{F}^{-1}$ is the inverse of $\mathcal F$.\\
$D$ is a pseudo-differential operator defined by
$\chi(D) f(x):=\mathcal F^{-1}(\chi(\xi)\hat f(\xi))(x)$
for a given a function $\chi(\xi)$ in the frequency space. \\
$S_q$ is a low frequency cut-off operator defined by
\begin{align*}
S_q{u}:=\mathcal{F}^{-1}\chi(2^{-q}\xi)\mathcal{F}u,
\quad
\forall q\in\mathbb{N} .
\end{align*}
$l^r(L^p)$ is the space of sequences of functions
$f_i(x)\in L^p$, $i\in\mathbb Z$,
such that $\{a_i\}_{i\in\mathbb Z}=\{\|f_i\|_{L^p}\}_{i\in\mathbb Z}$ satisfies
$(\sum_{i\in\mathbb Z} a_i^r)^{\frac 1 r<\infty}$. \\
An $S^k$-multiplier, $k\in\Rnum$, is a smooth operator
$A: \Rnum \to \Rnum$ such that,
$\forall \alpha\in \mathbb{N}^n$ and $\forall \xi \in \Rnum$,
$|\partial^\alpha{A(\xi)}|\leq C_\alpha(1+|\xi|)^{k-|\alpha|}$
holds for some constant $C_\alpha$.\\
$\langle \cdot,\cdot \rangle$ denotes the pairing between a normed space and its dual space (i.e., linear bounded functional space). \\

\begin{lemm}
For a pair of real functions $f_1,f_2$ on $\Rnum$:
\begin{align*}
\tfrac{1}{2} (\|f_1\|_{L^p}+\|f_2\|_{L^p}) \leq \|f_1 + i f_2\|_{L^p} \leq \|f_1\|_{L^p}+\|f_2\|_{L^p}
\end{align*}
\end{lemm}

\begin{proof}
The inequalities $|f_1|^p,|f_2|^p \leq (|f_1|^2 + |f_2|^2)^{\frac p2}$ directly imply
\begin{equation*}
\|f_1\|_{L^p}+\|f_2\|_{L^p} \leq 2\|\sqrt{|f_1|^2+|f_2|^2}\|_{L^p} = 2\||f_1+i f_2|\|_{L^p}.
\end{equation*}
Next, the triangle inequality $\sqrt{|f_1|^2+|f_2|^2}\leq |f_1|+|f_2|$ yields
\begin{equation*}
\|\sqrt{|f_1|^2+|f_2|^2}\|_{L^p} \leq \| |f_1|+|f_2|\|_{L^p}.
\end{equation*}
Then the Cauchy-Schwartz inequality $\| |f_1|+|f_2|\|_{L^p} \leq \|f_1\|_{L^p} + \|f_2\|_{L^p}$
completes the proof.
\end{proof}

\begin{prop}
(Littlewood-Paley decomposition) \\
There exists on $\Rnum$ a pair of smooth functions $(\chi,\varphi)$ valued in $[0,1]$,
with the following properties:\\
(1) $\chi$ is supported in the interval $\{\xi\in\Rnum: |\xi|\leq \frac{4}{3}\}$,
and $\varphi$ is supported in the ring $\{\xi\in\Rnum: \frac{3}{4}\leq|\xi|\leq \frac{8}{3}\}$; \\
(2) $\chi(\xi)+{\sum\limits_{q\geq 0}\varphi(2^{-q}\xi)}=1$ for all $\xi\in\Rnum$;\\
(3) $\supp(\varphi(2^{-q}\,\cdot)) \cap \supp(\varphi(2^{-q'}\,\cdot))=\emptyset$
if $|q-q'|\geq 2$,
and $\supp(\chi(\cdot)) \cap \supp(\varphi(2^{-q}\,\cdot))=\emptyset$
if  $q\geq 1$.
\end{prop}

For all $u \in \mathcal{S}'$,
we can define the nonhomogeneous dyadic blocks as follows.
Let
\begin{align*}
\Delta_q{u} & := 0 \text{ if } q\leq -2,
\\
\Delta_{-1}{u} & := \chi(D)u=\mathcal{F}^{-1}(\chi \mathcal{F}u),
\\
\Delta_q{u} & := \varphi(2^{-q}D)u=\mathcal{F}^{-1}(\varphi(2^{-q}\cdot)\mathcal{F}u) \text{ if } q\geq 0 .
\end{align*}
Then
\begin{align*}
u={\sum\limits_{q\in\mathbb{Z}}}\Delta_q{u}\quad \in \mathcal{S}'(\Rnum)
\end{align*}
where the right-hand side is called the nonhomogeneous Littlewood-Paley decomposition of $u$.

\begin{rema}\label{rem:S.op}
(i) The low frequency cut-off operator obeys
\begin{align*}
S_q u = \sum\limits_{p=-1}^{q-1}\Delta_p{u} .
\end{align*}
(ii) The Littlewood-Paley decomposition obeys
\begin{align*}
& \Delta_p\Delta_q{u}= 0 \quad\text{ if } |p-q|\geq2 ,
\\
& \Delta_q(S_{p-1}u \Delta_p{v})= 0 \quad\text{ if } |p-q|\geq5 ,
\end{align*}
for all $u,v \in \mathcal{S}'(\Rnum)$.
These properties describe quasi-orthogonality in $L^2$ in the following sense:
the symbols of $\Delta_q$ and $\Delta_p$ are $\varphi(2^{-q}\xi)$ and $\varphi(2^{-p}\xi)$ respectively, which obey
$\int_{\mathbb R}\varphi(2^{-q}\xi)\varphi(2^{-p}\xi)\d\xi=0$ for $|q-p|\geq 2$.
\\
(iii) Young's inequality implies, for all $1\leq p\leq\infty$,
\begin{align*}
\|\Delta_q{u}\|_{L^p}, \|S_q{u}\|_{L^p}\leq C\|u\|_{L^p}
\end{align*}
where $C$ is a positive constant independent of $q$.
\end{rema}

\begin{defi}
(Besov spaces) \\
Let $s\in\Rnum$ and $1\leq p,r\leq\infty$.
The nonhomogeneous Besov space is defined by
$B^s_{p,r}(\Rnum):= \{f \in \mathcal{S}'(\Rnum):\|f\|_{B^s_{p,r}}< \infty\}$
in terms of the Besov norm
\begin{align*}
\|f\|_{B^s_{p,r}}:= \|2^{qs}\Delta_q{f}\|_{l^r(L^p)}=\big\|(2^{qs}\|\Delta_q{f}\|_{L^p})_{q\geq-1}\big\|_{l^r}
\end{align*}
for $s<\infty$.
In the case $s=\infty$,
the space is defined by $B^\infty_{p,r}:= \bigcap\limits_{s\in\Rnum}B^s_{p,r}$.
\end{defi}

In the following lemma, we list some important properties of Besov spaces.
\begin{lemm}\label{lem:Besov.spaces}
Suppose that $s\in\Rnum$, 1$\leq p, r, p_{i},r_{i}\leq\infty$, $i=1, 2$.
\\
(i) Basic properties:
$B^s_{p,p}=W^{s,p}$;
$B^s_{2,2}=H^s$.
$B^s_{p,r}$ is a Banach space which is continuously embedded in $\mathcal{S}'$.
\\
(ii) Density:
$\C^{\infty}_c$ is dense in $B^s_{p,r}$ iff $1\leq p,r < \infty$.
\\
(iii) Embedding:
$B^s_{p_1,r_1}\hookrightarrow B^{s-({\frac1{ p_1}}-{\frac1{ p_2}})}_{p_2,r_2}$
if $p_1\leq p_2$ and $r_1\leq r_2$.
$B^{s_2}_{p,r_2}\hookrightarrow B^{s_1}_{p,r_1}$ is locally compact if $s_1 < s_2$.
$B^s_{p,r}\hookrightarrow L^{\infty}$
if $s>\frac{1}{p}$ or if $s\geq  \frac{1}{p}$ and $r=1$.
\\
(iv) Algebraic properties:
$B^s_{p,r}\bigcap L^\infty$ is an algebra if $s>0$.
$B^s_{p,r}$ itself is an algebra if $s>\frac{1}{p}$ or if $s\geq \frac{1}{p}$ and $r=1$.
\\
(v) Complex interpolation:
\begin{align}\label{interpolation1}
\|f\|_{B^{\theta {s_1}+(1-\theta){s_2}}_{p,r}}\leq \|f\|^{\theta}_{B^{s_1}_{p,r}}\|f\|^{1-\theta}_{B^{s_2}_{p,r}},
\quad \forall  u\in B^{s_1}_{p,r}\cap B^{s_1}_{p,r},\quad \forall  \theta \in [0,1].
\end{align}
(vi) Logarithm interpolation:
for all $s\in \Rnum$ and $\varepsilon >0$, there exists a constant $C$ such that
for any $u$ in $B^{s+\varepsilon}_{p,\infty}$ and $f$ in $B^{\varepsilon}_{\infty,\infty}$:
\begin{align}
& \|u\|_{B^s_{p,1}}\leq C (1+1/\varepsilon) \|u\|_{B^s_{p,\infty}}
\ln\Big(e+ \|u\|_{B^{s+\varepsilon}_{p,\infty}}/\|u\|_{B^{s}_{p,\infty}}\Big),
\quad 1\leq p \leq \infty;
\label{interpolation2}
\\
& \|f\|_{L^\infty}\leq C (1/\varepsilon) \|u\|_{B^0_{\infty,\infty}}
\ln\Big(e+ \|f\|_{B^{\varepsilon}_{p,\infty}}/\|f\|_{B^{0}_{\infty,\infty}}\Big).
\label{interpolation3}
\end{align}
(vii) Fatou's lemma:
if $(u_n)_{n\in \mathbb{N}}$ is bounded in $B^s_{p,r}$
and $u_n \to u $ in $\mathcal{S}'$,
then $u\in B^s_{p,r}$ and
\begin{align*}
\|u\|_{B^s_{p,r}}\leq C\liminf\limits_{n\to \infty} \|u_n\|_{B^s_{p,r}} .
\end{align*}
(viii) For all $k\in\Rnum$, an $S^k$-multiplier is continuous from $B^s_{p,r}$ to $B^{s-k}_{p,r}$.
\end{lemm}

\begin{lemm}\label{Moser}
Moser-type inequalities\\
(i)
Let $s>0$, $1\leq p,r\leq \infty$, $u,v\in B^s_{p,r}\cap L^\infty$.
Then there exists a constant $C=C(s)$ such that
\begin{align*}
\|uv\|_{B^s_{p,r}}\leq C (\|u\|_{L^\infty}\|v\|_{B^s_{p,r}}+\|v\|_{L^\infty}\|u\|_{B^s_{p,r}}) .
\end{align*}
(ii)
Let $1\leq p,r \leq\infty$, $(s_1,s_2)\in\mathrm R^2$ such that
$s_1\leq s_2$, $s_1+s_2>\max(0,\frac2p-1)$,
and $s_2 >\frac{1}{p}$ if $r\neq1$, or $s_2=\frac{1}{p}$ if $r=1$.
Then there exists a constant $C=C(s_1,s_2,p,r)$ such that
\begin{align*}
\|uv\|_{B^{s_1}_{p,r}}\leq C \|u\|_{B^{s_1}_{p,r}}\|v\|_{B^{s_2}_{p,r}} .
\end{align*}
(iii)
For any  $u\in B^{\frac{1}{p}-1}_{p,\infty}(\Rnum)$ and $v\in B^{\frac{1}{p}}_{p,1}(\Rnum)$,
we have
\begin{align*}
\|uv\|_{B^{\frac{1}{p}-1}_{p,\infty}}\leq C \|u\|_{B^{\frac{1}{p}-1}_{p,\infty}}\|v\|_{B^{\frac{1}{p}}_{p,1}}
\end{align*}
\end{lemm}

\begin{prop}\label{testfunction}
Let $1\leq p,r\leq\infty $ and $s\in\Rnum$. \\
(i) For all $u\in B^s_{p,r}$ and $\phi\in B^{-s}_{p',r'}$,
\begin{align*}
(u,\phi) \longmapsto \sum\limits_{|j-i|\leq 1} \langle\Delta_ju,\Delta_{i}\phi\rangle
\end{align*}
defines a continuous bilinear functional on $B^s_{p,r} \times B^{-s}_{p',r'}$. \\
(ii) Denote by $Q^{-s}_{p',r'}$ the set of functions $\phi$ in $\mathcal{S}$ such that $\|\phi\|_{B^{-s}_{p',r'}}\leq1$.
If $u$ is in $\mathcal{S}'$, then
\begin{align*}
 \|u\|_{B^{s}_{p,r}}\leq C\sup\limits_{\phi\in Q^{-s}_{p',r'}}\langle u,\phi\rangle.
\end{align*}
\end{prop}

Now we state some useful results in the transport equation theory,
which are crucial to the proofs of our main theorems.

\begin{lemm}\label{apriori-estimates}
(A priori estimates in Besov spaces)\\
Let $1\leq p\leq p_1\leq \infty$, $1\leq r \leq \infty$, and
$\sigma\geq -\min (\tfrac{1}{p_1}, 1-\tfrac{1}{p})$ with strict inequality if $r<\infty$.
Assume that $f_0\in B^\sigma_{p,r}$, $F\in L^1(0,T; B^\sigma_{p,r})$,
and consider the linear transport equation
\begin{align*}
(*)\qquad
\partial_t f+v\,\partial_x f=F,
\quad
f|_{t=0} =f_0.
\end{align*}
Suppose $\partial_x v$ belongs to $L^1(0,T; B^{\sigma-1}_{p_1,r})$
for $\sigma> 1+{\frac{1}{p_1}}$, $r\neq1$,
or $\sigma =1+\frac1{p_1}$, $r=1$),
and otherwise $\partial_x v$ belongs to $L^1(0,T; B^{\frac1{p_1}}_{p_1,\infty}\bigcap L^\infty)$.
If $f\in L^\infty(0,T; B^\sigma_{p,r})\bigcap \C([0,T]; \mathcal{S}')$ solves $(*)$,
then there exists a constant $C$, depending only on $p$, $r$, and $\sigma$,
such that the following statements hold:\\
(i)
\begin{equation*}
\|f(t)\|_{B^\sigma_{p,r}}\leq \|f_0\|_{B^\sigma_{p,r}}\,+\, \int_0^t \|F(\tau)\|_{B^\sigma_{p,r}}\d \tau\,+\, C\int_0^t V'(\tau)\|f(\tau)\|_{B^\sigma_{p,r}} \d \tau
\end{equation*}
and hence
\begin{equation}\label{apriori}
\|f(t)\|_{B^\sigma_{p,r}}\leq e^{CV(t)} \left(\|f_0\|_{B^\sigma_{p,r}}\,+\, \int_0^t e^{-CV(\tau)} \|F(\tau)\|_{B^\sigma_{p,r}}\d \tau\right),
\end{equation}
with
\begin{equation*}
V(t)=\begin{cases}
\int_0^t \|\partial_x v(\tau)\|_{B^{\frac1{p_1}}_{p_1,\infty}\cap L^\infty}\d \tau,
& \sigma < 1+\frac 1{p_1}
\\
\int_0^t \|\partial_x v(\tau)\|_{B^{\sigma-1}_{p_1,r}}\d \tau,
& \sigma > 1+\frac1{p_1}, r\neq 1,
~\text{or}~
\sigma= 1+\frac1{p_1}, r=1
\end{cases} .
\end{equation*}
(ii) If $r<\infty$, then $f\in \C([0,T]; B^\sigma_{p,r})$.
If $r=\infty$, then $f\in \C([0,T]; B^{\sigma'}_{p,1})\cap \C_w([0,T];B^s_{p,\infty}(\Rnum))$  for all $\sigma'<\sigma$.
\\
(iii) If $f=v$, then for all $\sigma>0$ the estimate \eqref{apriori} holds
with $V'(t)=\|\partial_xv(t)\|_{L^\infty}$.
\end{lemm}

\begin{lemm}\label{Luowei}
\cite{LuoYin}
Let $1\leq p\leq\infty$ and $1\leq r \leq \infty$.
Assume $f_0\in B^{1+\frac{1}{p}}_{p,r}(\Rnum)$, $F\in L^1(0,T;B^{1+\frac{1}{p}}_{p,r}(\Rnum))$, and $v\in L^1(0,T;B^{2+\frac{1}{p}}_{p,r}(\Rnum))$
in the linear transport equation $(*)$.
If $f\in L^\infty(0,T;B^{1+\frac{1}{p}}_{p,r}(\Rnum)) $ solves $(*)$,
then
\begin{equation*}
\|f(t)\|_{B^{1+\frac{1}{p}}_{p,r}}\leq e^{CV(t)} (\|f_0\|_{B^{1+\frac{1}{p}}_{p,r}}\,+\, \int_0^t e^{-CV(\tau)} \|F(\tau)\|_{B^{1+\frac{1}{p}}_{p,r}}\d \tau),
\end{equation*}
with $V(t)=\int_0^t\|v\|_{B^{2+\frac{1}{p}}_{p,r}(\Rnum)}\d \tau$ and  $C=C(p,r)$.
\end{lemm}

\begin{lemm}\label{LiYin2}
\cite{LiYin2}
If $\sigma>0$, then there exists a constant $C = C(p,r,\sigma)$ such that
\begin{equation*}
\|f(t)\|_{B^\sigma_{p,r}}\leq \|f_0\|_{B^\sigma_{p,r}}
+\int_0^t \|F(\tau)\|_{B^\sigma_{p,r}}\d \tau+C\int_0^t \big(\|f(\tau)\|_{B^\sigma_{p,r}}\|v_x\|_{L^\infty}+
\|f(\tau)\|_{L^\infty}\|v_x\|_{B^\sigma_{p,r}}\big)\d\tau.
\end{equation*}
\end{lemm}

\begin{lemm}\label{regularity}
(Existence and uniqueness)
Let $p,r,\sigma,f_0$ and $F$ be as in the statement of Lemmas \ref{apriori-estimates}--\ref{Luowei}.
Assume that $v\in L^k(0,T; B^{-M}_{\infty,\infty})$ for some $k>1$ and $M>0$,
and that
\begin{align*}
\partial_x v \in \begin{cases}
L^1(0,T; {B^{\sigma-1}_{p,r}}),
& \sigma> 1+\tfrac{1}{p}, r\neq 1, \text{ or } \sigma=1+\tfrac{1}{p}, r=1
\\
L^1(0,T; B^{\frac{1}{p}}_{p,\infty}\cap L^\infty),
& \sigma<1+{\frac{1}{p}}
\end{cases}
\end{align*}
and $v\in L^1(0,T;B^{\sigma+1}_{p,r}(\Rnum))$ if $\sigma= 1+{\frac{1}{p}}$, $r>1$.
Then $(*)$ has a unique solution
\begin{align*}
f\in L^{\infty}(0,T; B^\sigma_{p,r})\bigcap\,\big(\bigcap\limits_{\sigma'<\sigma} \C([0,T]; B^{\sigma'}_{p,1})\big)
\end{align*}
and the inequalities of Lemmas~\ref{apriori-estimates}--\ref{Luowei} hold.
Moreover, if $r<\infty$, then $f\in \C([0,T]; B^\sigma_{p,r})$.
\end{lemm}

\end{document}